\documentclass[12pt,a4paper]{amsart} 
\usepackage{times}
\usepackage{amssymb,amsmath,amsthm,amsaddr}
\theoremstyle{defin}
\usepackage{amscd}
\usepackage{float}
\usepackage{tikz-cd}
\usepackage{tikz}
\usetikzlibrary{fit,backgrounds,positioning}
\usetikzlibrary {fit,shapes.geometric} 
\usepackage{mathtools}
\usepackage{hyperref}
\usepackage[utf8]{inputenc}
\usepackage[T1]{fontenc}  
\usepackage{graphicx}
\usepackage{vmargin} 
\usepackage{url} 
\usepackage{tikz-3dplot}
\usepackage{subcaption}
\usetikzlibrary{arrows.meta, positioning}

\tikzset{
	>=Stealth,
	node/.style={circle, draw, minimum size=6mm, inner sep=0pt}
}
\usepackage{enumitem}
\usepackage{wrapfig}
\usepackage{ae}
\usepackage{here} 
\usepackage{tikz} 
\usetikzlibrary{3d}

\newtheorem{prop}{Proposition}[section] 
\newtheorem{ex}{Example}
\newtheorem{defin}{Definition}[section]
\newtheorem{remarque}{Remark}[section] 
\newtheorem{exemple}{Example}[section] 
\newtheorem{lem}{Lemma}[section] 
\newtheorem{cor}{Corollary}[section]

\newtheorem{prob}{Open Problem}[section]
\newtheorem{rem}{Remark}[section]
\newcommand{\seqnum}[1]{\href{https://oeis.org/#1}{\rm \underline{#1}}}
 
\newcommand{\Fix}{{\rm Fix}}
\def\Aut{\mathrm{Aut}}

\def\Id{\mathrm{Id}}

\def\Aut{\mathrm{Aut}}

\def\Aut{\mathrm{Aut}}

\def\Id{\mathrm{Id}}

\def\Frob{\mathrm{Frob}}

\def\Aut{\mathrm{Aut}}

\def\Ens{\mathrm{\mathbf{Ens}}}

\def\Aut{\mathrm{Aut}}

\def\Ens{\mathrm{\mathbf{Ens}}}

\newcommand{\charac}{\raise 2pt\hbox{\large$\chi$}}  
\newcommand{\phirac}{\raise 2pt\hbox{\large$\phi$}} 

\begin{document}
	\title{Species, Symmetric Functions, and Kronecker Product}
    \author{Josaphat Baolahy}
    \address{Department of Mathematics, Faculty of Science, University of Fianarantsoa, Madagascar}
    \email{japhtbaolah@gmail.com}
    \thanks{The first author was supported by the IMU GRAID program.} 

    \author{Randrianirina Benjamin}
    \address{Department of Mathematics, Faculty of Science, University of Fianarantsoa, Madagascar}
    \email{rabezarand@gmai.com}
    \date{\today}
    \begin{abstract}
We study two new families of symmetric functions arising from a species-theoretic construction motivated by cycle structure. For each partition of $n$, we define two combinatorial species that decompose into molecules indexed by the same partition, giving rise to two corresponding basis of the homogeneous symmetric functions of degree $n$. We prove that each of these families forms a basis by exhibiting explicit cycle-index formulas and triangular transition matrices to the power-sum basis. Using these constructions, we generalize a classical result describing the Kronecker (Hadamard) product in the homogeneous basis to the two new settings. In particular, we show that the categories generated by these species are closed under the Kronecker product, and that the product of two basis elements expands with nonnegative integer coefficients. Our results provide a new combinatorial framework for studying the Kronecker product and suggest avenues toward interpreting its structure constants.
\end{abstract}

	\maketitle
\section{Introduction}
The Kronecker product of symmetric functions, and its notoriously difficult structure constants for the Schur basis, is a central problem in algebraic combinatorics. This paper explores this problem through the lens of combinatorial species, where other natural basis yield positive and computable, though still mysterious, coefficients. Let $V$ and $W$ be two representations of the general linear group $\mathrm{GL}_n(\mathbb{C})$, with respective characters $\chi_V$ and $\chi_W$. The tensor product $V\otimes W$ is itself a representation of $\mathrm{GL}_n(\mathbb{C})$, whose character is given by:
$$
(\chi_V\star \chi_W)(M) = \chi_V(M)\cdot \chi_W(M), \quad \mbox{ for all $M\in \mathrm{GL}_n$}
$$ 
where the product on the right is complex number multiplication. This product $\star$ is the famous Kronecker product of symmetric functions.
A central problem regarding the Kronecker product arises from the basis of Schur functions. Recall that if $\lambda$ is a partition of $n$, the Schur function $s_{\lambda}$ is the character of the Weyl module $W^{\lambda}$, and the family $\left\{ s_{\lambda} \right\}_{\lambda\vdash n}$ forms a basis of the algebra of symmetric functions. Since the tensor product $W^{\alpha}\otimes W^{\beta}$ is a representation of $\mathrm{GL}_n$, its character is Schur-positive; that is,
$$
s_{\alpha}\star s_{\beta} = \sum_{\mu}g_{\alpha,\beta}^{\mu} s_{\mu} \quad \mbox{ where $g_{\alpha,\beta}^{\mu}\in \mathbb{N}$}.
$$
However, finding a combinatorial interpretation of the coefficients $g_{\alpha,\beta}^{\mu}$ remains a major open problem.

The work of A.M. Garsia and J. Remmel \cite{GR} provided a solution to an analogous question for the basis of complete homogeneous symmetric functions $\{h_{\lambda}\}_{\lambda\vdash n}$:
\[
h_{\alpha}\star h_{\beta} \;=\; \sum_{\mu} NM_{\alpha,\beta}^{\mu}\, h_{\mu}
\]
where $NM_{\alpha,\beta}^{\mu}$ counts the number of non-negative integer matrices such that the sum of the $i$-th row is equal to $\alpha_i$, the sum of the $j$-th column is equal to $\beta_j$, and the entries of the matrix (when arranged in non-increasing order) form the partition $\mu$.
Extending this result to other basis of the ring of symmetric functions offers new insights into the structure of the Kronecker product. While this does not solve the open problem regarding the Kronecker coefficients directly, it represents significant progress in this exploration. The primary contribution of this paper is the introduction of two new basis of $\Lambda_n$, denoted $\left\{\mathbf{C}_{\alpha}(\mathbf{z})\right\}_{\alpha \vdash n}$ and $\left\{\mathbf{K}_{\alpha}(\mathbf{z})\right\}_{\alpha\vdash n}$, for which the Kronecker coefficients are explicit. Readers unfamiliar with the theory of symmetric functions are referred to the classical texts~\cite{IG}\cite{SR}.
These new basis are the \textit{cycle index series} of families of molecular species defined as follows. For any partition $\alpha = (i_1^{\alpha_1},\ldots, i_m^{\alpha_n})$, we set
\[
\mathbf{C}_{\alpha} \;:=\; \left(X^n/\langle \sigma \rangle\right),
\mbox{ where $\sigma$ is a permutation of cycle type $\alpha$, denoted $\lambda(\sigma)=\alpha$ }\]
and 
\[
\mathbf{K}_{\alpha} \;:=\; \mathbf{C}_{i_1^{\alpha_1}} \cdots \mathbf{C}_{i_m^{\alpha_m}}.
\]
Their corresponding cycle index series are given by
\[
\mathbf{C}_{\alpha}(\mathbf{z}) \;=\; \frac{1}{o(\sigma)} \sum_{k=1}^{o(\sigma)} p_{\lambda(\sigma^k)}(\mathbf{z})
\qquad \text{and} \qquad
\mathbf{K}_{\alpha}(\mathbf{z}) \;=\; \mathbf{C}_{i_1^{\alpha_1}}(\mathbf{z}) \cdots \mathbf{C}_{i_m^{\alpha_m}}(\mathbf{z}).
\]
The complete homogeneous symmetric functions provide a classical example of symmetric functions derived from molecular species. Indeed, they can be expressed in terms of the cycle index series as follows:
\[
h_{\lambda} = h_{\lambda_1} \cdots h_{\lambda_k} = Z_{\mathbf{E}_{\lambda_1} \cdots \mathbf{E}_{\lambda_k}}, \quad 
\]
where $\lambda = (\lambda_1, \ldots, \lambda_k)$ is a partition, and $\mathrm{E}$ is the species of sets. Here, $Z_{\mathrm{S}}$ denotes the cycle index series of a species $\mathrm{S}$. 
To emphasize the structural parallels between these basis, we simultaneously study the three families of molecular species $\{\mathbf{E}_{\alpha}\}_{\alpha\vdash n}$, $\{\mathbf{C}_{\alpha}\}_{\alpha\vdash n}$, and $\{\mathbf{K}_{\alpha}\}_{\alpha\vdash n}$, where $\alpha$ is a partition of $n$. The result of A.M. Garsia and J. Remmel \cite{GR} (see \cite{GR}) corresponds to the case of $\{\mathbf{E}_{\alpha}\}_{\alpha\vdash n}$, recast here in the language of combinatorial species theory.
Our second main contribution is to extend this result to the basis $\{\mathbf{C}_{\alpha}(\mathbf{z})\}_{\alpha\vdash n}$ and $\{\mathbf{K}_{\alpha}(\mathbf{z})\}_{\alpha\vdash n}$ by establishing the following decomposition formulas for the $\star$-product (the Hadamard product of species generating functions):
\[
	\mathbf{C}_{\alpha}(\mathbf{z})\star \mathbf{C}_{\beta}(\mathbf{z}) \; = \; \sum_{\mu} b_{\alpha,\beta}^{\mu}\, \mathbf{C}_{\mu}(\mathbf{z}) \quad \text{ and }\quad 
	\mathbf{K}_{\alpha}(\mathbf{z})\star \mathbf{K}_{\beta}(\mathbf{z}) \; = \; \sum_{\mu} j_{\alpha,\beta}^{\mu}\, \mathbf{K}_{\mu}(\mathbf{z}).
\]
In terms of combinatorial species, these identities translate into the existence of the following natural isomorphisms:
\[
\mathbf{C}_\alpha \times \mathbf{C}_\beta = \sum_{\mu \vdash n} b_{\alpha,\beta}^{\mu}\, \mathbf{C}_{\mu},
\qquad \mbox{ and }\quad 
\mathbf{K}_\alpha \times \mathbf{K}_\beta = \sum_{\mu \vdash n} j_{\alpha,\beta}^{\mu}  \, \mathbf{K}_{\mu}.
\] 
To interpret this work within the framework of species theory, we introduce three subcategories of the category of combinatorial species $\mathrm{Esp}$, whose elements are positive linear combinations of $\mathbf{E}_{\mu}$, $\mathbf{C}_{\mu}$, and $\mathbf{K}_{\mu}$, respectively. We denote them by
\[
\mathrm{H}\mathrm{Esp}, 
\qquad 
\mathrm{C}\mathrm{Esp}, 
\qquad 
\mathrm{K}\mathrm{Esp}.
\]
The objects within these subcategories take the form:
\[
F=\sum_{\mu} a_{\mu} \mathbf{E}_{\mu}, 
\qquad 
G=\sum_{\mu} b_{\mu} \mathbf{C}_{\mu}, 
\qquad 
H=\sum_{\mu} d_{\mu} \mathbf{K}_{\mu}.
\]
To provide a combinatorial proof of closure under the Hadamard product, we utilize standard results from combinatorial species theory. This approach allows us to reformulate the proof of A.M. Garsia and J. Remmel \cite{GR} in the language of species. The closure of the other two categories is established similarly, though in distinct contexts. Moreover, for all $r \geq 1$, we satisfy the identities:
\[
\mathbf{C}_{\alpha}(\mathbf{z})(p_1^r) = \mathbf{C}_{\alpha^r}(\mathbf{z}), 
\qquad \mbox{and} \qquad
\mathbf{K}_{\alpha}(\mathbf{z})(p_1^r) = \mathbf{K}_{\alpha^r}(\mathbf{z}),
\]
where $\alpha^r$ denotes the partition obtained from $\alpha$ by repeating each part $r$ times.
 
The paper is organized as follows. In Section 2, we review the fundamental theory of combinatorial species and the general concept of molecular species. Section 3 is dedicated to the detailed construction and study of the three specific families of molecules: set molecules, cyclic molecules of the first kind, and cyclic molecules of the second kind. In Section 4, we explicitly define the associated symmetric functions $\{C_{\alpha}(\mathbf{z})\}$ and $\{K_{\alpha}(\mathbf{z})\}$ and establish them as basis of the ring of symmetric functions by exhibiting triangular transition matrices to the power-sum basis. Finally, in Section 5, we introduce the subcategories $\mathrm{CEsp}$ and $\mathrm{KEsp}$ and prove our main result regarding their closure under the Hadamard (Kronecker) product, extending the classical framework of A.M. Garsia and J. Remmel \cite{GR}.
\section{Combinatorial Species and Molecular Species}
In this section, we introduce the notions of combinatorial species and molecular species. By analogy with the structure of matter, which is an assembly of molecules, combinatorial species are finite sums of molecular species. Beyond this fundamental property, this work demonstrates that there exist three families of molecular species whose cycle index series form a $\mathbb{Q}$-basis of the ring of symmetric functions. We begin with the formal definition of a combinatorial species.
\begin{defin}
A combinatorial species $F$ is a functor from the category of finite sets and bijections $\mathbb{B}$ to the category of finite sets and functions $\Ens$, i.e.
$$
F : \mathbb{B} \to \Ens.
$$
\end{defin}
It is a procedure that constructs a set of structures $F[U]$ from a finite set $U$. The elements of $F[U]$ are called \textit{the $F$-structures on U}.
Some classical examples include:
\begin{itemize}
	\item the species of sets $\mathrm{E}$, which assigns the set $\mathrm{E}[U] = \{U\}$ to each finite set $U$.
	\item the species of subsets $\mathcal{P}$, which assigns the set $$\mathcal{P}[U] = \{A\;|\;A\subset U\}$$ to each finite set $U$.
	\item the species of partitions $\mathrm{Par}$, which assigns the set 
	$$
	\mathrm{Par}[U] = \{\pi \;|\; \pi \mbox{ is a partition of the set $U$}\}
	$$
	to each finite set $U$.
	\item the species of permutations $\mathcal{S}$, which assigns the set 
	$$
	\mathcal{S}[U] = \{\sigma \;|\; \sigma \mbox{ is a permutation of the set $U$}\}
	$$
	to each finite set $U$.
	\item the species of simple graphs $\mathcal{G}$, which assigns the set 
	$$
	\mathcal{G}[U] = \{(U,E) \;|\; (U,E) \mbox{ is a simple graph with vertex set $U$ and edge set $E$}\}
	$$
	to each finite set $U$.
\end{itemize} 
Given a finite set $U$, we can associate to each species $F$ its unlabelled version $\tilde{F}$, defined as the quotient of $F[U]$ by the following equivalence relation:
$$
\forall s,t \in F[U],\; s\sim t \quad \Leftrightarrow \quad \exists \sigma : U\xrightarrow{\simeq} U,\; F[\sigma]s = t.
$$
The elements of $\tilde{F}$ are called the \textit{isomorphism types of} $F$ or the unlabelled structures of $F$. They are obtained by ignoring the labels of the elements in $U$ on the $F$-structures. We define the cycle index series of a species $F$ by
$$
	Z_{F}(p_1,\ldots,) = \sum_{n\geq 0} \frac{1}{n!}\sum_{\sigma\in S_n} \#\Fix[F[\sigma]]p_{\lambda(\sigma)}(x_1,\ldots), 
$$
where $p_{\lambda(\sigma)}$ is the power sum associated with the cycle type of the permutation $\sigma$. This concept was first introduced by George Pólya to study the enumeration of objects under symmetry (see \cite{GP}).
In this work, we are particularly interested in molecular species, a specific class of species defined below. They are characterized by having only one isomorphism type, and every species can be decomposed into a sum of products of molecular species (see \cite{FGP} \cite{JOYAL19811}). Our work highlights another significant property for some of these molecular species: the family of their cycle index series forms a $\mathbb{Q}$-basis of the ring of homogeneous symmetric functions $\Lambda_n$. 

\begin{defin}
Let $H$ be a subgroup of $S_n$. We define the associated molecular species by setting, for every finite set $U$,
$$
\left(X^n/H\right)[U] = \{\lambda H\;\big|\; \lambda :  [n]\xrightarrow{\simeq} U \} \quad \mbox{ where }\quad \lambda H = \{\lambda f\;|\; f\in H\},
$$
and for every bijection $\sigma:U\to V$,
$$
\left(X^n/H\right)[\sigma]\lambda H = (\sigma \lambda) H.
$$
The operation is the composition of functions.
\end{defin}
\begin{remarque}
	The species $\left(X^n/H\right)$ has only one isomorphism type. Indeed, since the action of $H$ on $S^n/H$ is transitive, for every $\lambda_1\, H, \lambda_2\, H\in \left(X^n/H\right)[U]$ there exists a $\sigma\in S_{U}$ such that $\lambda_1\, H = (\sigma\lambda_2)\, H$ (see \cite{FGP}).  
\end{remarque}
In general, we have the following results:
\begin{ex}[\cite{FGP}]\leavevmode
	\begin{enumerate}
		\item[(a)] If $\alpha = (\alpha_1,\alpha_2\ldots,\alpha_m)$ and $S_{\alpha}=S_{\alpha_1}\times S_{\alpha_2}\times\cdots \times S_{\alpha_m} $, then $$\left(X^n/S_{\alpha}\right)=\mathbf{E}_{\alpha_1}\cdot \mathbf{E}_{\alpha_2}\cdots \mathbf{E}_{\alpha_m}=\mathbf{E}_{\alpha}.$$
		\item[(b)] Let $\sigma \in S_n$. Then
		$\left(X^n/\langle \sigma\rangle\right)[n] = \{\lambda \langle \sigma \rangle\;\big|\; \lambda:[n]\to U\}
		$. These are referred to as cyclic species of the first kind.
		\item[(c)] Let $\sigma$ be a permutation whose cycle decomposition is of the form 
		$$
		\sigma = \sigma_1^{1}\cdots \sigma_1^{m_1}\sigma_2^{1}\cdots \sigma_2^{r_2}\cdots \sigma_n^{1}\cdots \sigma_m^{r_m}		
		$$ 
		and let
		$G_{\sigma} = \langle\sigma_1^{1}\cdots \sigma_1^{r_1} \rangle \times \langle\sigma_2^{1}\cdots \sigma_2^{r_2}\rangle\times \cdots \times\langle\sigma_m^{1}\cdots \sigma_n^{r_m}\rangle$,
		where the $\sigma_i$ are cycles of length $i$.
	We define another molecular species by
		$$
		X^n/G_{\sigma}.
		$$
	\end{enumerate}
\end{ex}
These molecular species are the fundamental objects of this work. We will later provide a more detailed study of them. We recall the following classical property.
\begin{lem}\label{espconj}\cite{FGP}
	Let $H,K\leq S_n$. Then
	$$
	\left(X^n/H\right)=\left(X^n/K\right) \Leftrightarrow \mbox{$H$ and $K$ are conjugate}.
	$$
\end{lem}
\begin{prop}\cite{FGP}
	The cycle index series of $\left(X^n/H\right)$ is given by 
	$$
	Z_{\left(X^n/H\right)} = \frac{1}{|H|}\sum_{\sigma\in H} p_{\lambda(\sigma)}.
	$$
\end{prop}
\begin{ex}\leavevmode
\begin{enumerate}
	\item Let $\alpha = (\alpha_1,\alpha_2\cdots,\alpha_m)$ be a partition of a nonnegative integer $n$. Then
	$$
	Z_{X^n/S_{\alpha}} = \prod_{i=1}^m \frac{1}{\alpha_i!}\sum_{\sigma\in S_{\alpha_i}} p_{\lambda(\sigma)}=\prod_{i=1}^m h_{\alpha_i} = h_{\alpha}.$$
	\item Let $\sigma$ be a permutation of $S_n$. Then
	$$
	Z_{\left(X^n/\langle\sigma\rangle\right)}=\frac{1}{o(\sigma)} \sum_{k=1}^{o(\sigma)} p_{\lambda(\sigma^k)}.
	$$
	\item Let $\sigma$ be a permutation whose cycle decomposition is of the form 
	$$
	\sigma = \sigma_1^{1}\cdots \sigma_1^{r_1}\sigma_2^{1}\cdots \sigma_2^{r_2}\cdots \sigma_m^{1}\cdots \sigma_m^{r_m}.
	$$
	Then
	$$
	Z_{(X^n/G_{\sigma})} = \prod_{i=1}^m Z_{(X^{r_i}/\langle\sigma_1^{1}\cdots \sigma_1^{r_i} \rangle)}.
	$$
\end{enumerate}
\end{ex}
Molecular species behave very well under operations on species.
\begin{prop}\label{propprod}\cite{FGP}
	Let $H$ and $K$ be two subgroups of $S_n$ and $S_m$ respectively. Then:
	\begin{enumerate}
		\item $\left(X^n/H\right)\cdot \left(X^m/K\right) \;=\; \left(X^{n+m}/(H\times K)\right)$,
		\item $\left(X^n/H\right)\times \left(X^n/K\right) = \sum\limits_{\tau \in H\backslash S_n \slash K} \left(X^n/H\cap \tau K\tau^{-1}\right)$.
	\end{enumerate}
\end{prop}
While both properties are indispensable to our work, the latter is crucial for establishing closure under the Kronecker product. These two operations correspond to classical operations on symmetric functions.
\begin{prop}[\cite{FGP}]\label{symesp}
		Let $H$ and $K$ be two transitive subgroups of $S_n$ and $S_m$ respectively. Then:
	\begin{enumerate}
		\item $Z_{\left(X^{n}/H\right)\cdot \left(X^m/K\right)} \;=\; Z_{\left(X^n/H\right)}\cdot Z_{\left(X^m/K\right)}$,
		\item $Z_{\left(X^n/H\right)\times \left(X^m/K\right)} \;=\ Z_{\left(X^n/H\right)}\star Z_{\left(X^m/K\right)}$
		where $\star$ is the Kronecker product of symmetric functions.
	\end{enumerate}
\end{prop} 
We will now study these three types of molecular species in turn: set molecules, cyclic molecules of the first kind, and cyclic molecules of the second kind.
\section{Molecules}
This section is dedicated to the definition of the molecular species that constitute the central objects of this work. We begin by defining set molecules, which provide the combinatorial framework for the classical results of Garsia and Remmel. Subsequently, we introduce two new families of molecules: cyclic molecules of the first kind and cyclic molecules of the second kind.
\subsection{Set molecules}
\begin{defin}
Let $\alpha = (\alpha_1,\alpha_2,\ldots,\alpha_m)$ be a partition of $n$. We define \textit{the species of lists of sets} of shape $\alpha$, denoted $\mathbf{E}_{\alpha}$, as the product of the species $\mathbf{E}_{\alpha_1}, \mathbf{E}_{\alpha_2},\cdots, \mathbf{E}_{\alpha_m}$, i.e.,
$$
\mathbf{E}_{\alpha} := X^n/S_{\alpha_1}\times \cdots \times X^n/S_{\alpha_m} = \mathbf{E}_{\alpha_1}\cdot \mathbf{E}_{\alpha_2} \cdots \mathbf{E}_{\alpha_m}.
$$
An $\mathbf{E}_{\alpha}$-structure on a set $U$ of size $n$ can be viewed as an equivalence class of sequences of words $$\overline{w_1}\,\overline{w_2}\,\ldots \overline{w_m} = \{(\sigma_1 w_1) \, (\sigma_2w_2)\, \ldots (\sigma_m w_m)\;|\; (\sigma_1,\sigma_2,\ldots, \sigma_m)\in S_{\alpha}\}$$
where each $w_i$ has length $\alpha_i$ for each $i\in[m]$.
\end{defin} 
In what follows, we will represent the $\mathbf{E}_{\alpha}$-structure as a word to facilitate a smoother discussion of the other molecules.
\begin{exemple}
Here are examples of $\mathbf{E}_{332}$-structures on $\{1,2,3,4,5,6,7,8\}$:
\begin{itemize}
	\item $\overline{134}\,\overline{528}\,\overline{57} = \overline{143}\,\overline{582}\,\overline{57} = \overline{413}\,\overline{825}\,\overline{75} = \cdots $,
	\item $\overline{528}\,\overline{134}\,\overline{57} = \overline{258}\,\overline{143}\,\overline{57} = \overline{825}\,\overline{314}\,\overline{75} = \cdots $.
\end{itemize}
\end{exemple}	

\subsection{Cyclic molecules of the first kind}
Throughout this section, given a partition $\alpha = (\alpha_1, \alpha_2, \ldots, \alpha_k)$ of an integer $n$, we define the \textit{standard permutation of shape $\alpha$}, denoted $\sigma_\alpha$, as the permutation in cycle notation obtained by filling the cycles corresponding to the parts of $\alpha$ with the elements of the set $[n] = \{1, 2, \ldots, n\}$ in increasing order. 
For example:
\begin{enumerate}
    \item If $\alpha = (4, 2, 1)$, then $n = 4+2+1 = 7$, and the standard permutation is:
    $$\sigma_{421} = (1, 2, 3, 4)(5, 6)(7)$$
    \item If $\alpha = (3, 3, 2)$, then $n = 3+3+2 = 8$, and the standard permutation is:
    $$\underline{332} = (1, 2, 3)(4, 5, 6)(7, 8).$$
\end{enumerate}
\begin{defin}
	Let $\alpha=(\alpha_1,\alpha_2,\ldots,\alpha_k)$ be a partition of $n$ and $\sigma_\alpha$ the standard permutation of shape $\alpha$. We define the species $\mathbf{C}_{\alpha}$ of lists invariant under the subgroup generated by $\sigma_\alpha$ as
	$$
	\mathbf{C}_{\alpha} = X^n/\langle \sigma_\alpha \rangle.
	$$ 
\end{defin}
Any permutation $\sigma$ of cycle type $\alpha$ yields a species isomorphic to $\mathbf{C}_{\alpha}$ by Lemma \ref{espconj}, i.e., $\mathbf{C}_{\alpha}=X^n/\langle \sigma \rangle$. Indeed, two cyclic subgroups generated by permutations are conjugate if and only if the underlying permutations have the same cycle type.
A $\mathbf{C}_{\alpha}$-structure on a set $U$ of size $n$ can be viewed as an equivalence class of sequences of words $\overline{[w_1]}\,\overline{[w_2]}\,\ldots \overline{(w_m)}=\{\sigma(w_1w_2\ldots w_m) \;|\; \sigma \in \langle\sigma_\alpha\rangle\}$ on $U$ of shape $\alpha$, where $w_i$ has length $\alpha_i$.
\begin{exemple}
For example, the object $\overline{[153]}\,\overline{[42]}\,\overline{[6]}=\{153426,531246,315426\}=\overline{[531]}\,\overline{[24]}\,\overline{[6]}=\overline{[315]}\,\overline{[42]}\,\overline{[6]}$ is a $\mathbf{C}_{321}$-structure on $\{1,2,3,4,5,6\}$.
\end{exemple}
\begin{prop}\label{iso1}
	Let $\alpha = (\alpha_1,\alpha_2,\ldots,\alpha_m)$ be a partition of $n$, and let $r\in\mathbb{N}$. We have the following isomorphism of species:
	\begin{equation*}
		\mathbf{C}_{\alpha}\circ X^r = \mathbf{C}_{\alpha^r}
	\end{equation*}
	where $\alpha^r = (\alpha_1, \ldots, \alpha_1, \alpha_2, \ldots, \alpha_2, \ldots, \alpha_m, \ldots, \alpha_m)$ is the partition of $nr$ containing exactly $r$ copies of each part $\alpha_i$.
\end{prop}
\begin{proof}
	We construct a natural isomorphism $T : \mathbf{C}_{\alpha^r} \to \mathbf{C}_{\alpha} \circ X^r$.
	
	Let $U$ be a finite set of size $nr$. A $\mathbf{C}_{\alpha^r}[U]$-structure is the equivalence class (under the action of the standard permutation of type $\alpha^r$) of a concatenation of $m \cdot r$ words. We can group these as $m$ blocks, each containing $r$ words of length $\alpha_i$. Let $s \in \mathbf{C}_{\alpha^r}[U]$ be such a structure:
	$$
	s = \prod_{i=1}^m \left( \prod_{j=1}^r \overline{[w_{j,i}^{1}w_{j,i}^{2}\ldots w_{j,i}^{\alpha_i}]} \right)
	$$
	A $(\mathbf{C}_{\alpha} \circ X^r)[U]$-structure is a $\mathbf{C}_{\alpha}$-structure (a word of $m$ "super-letters") where each "super-letter" is an $X^r$-structure (an $r$-tuple of elements from $U$).	
	We define the transformation $T_U : \mathbf{C}_{\alpha^r}[U] \to (\mathbf{C}_{\alpha} \circ X^r)[U]$ by "transposing" the elements. The map takes the $k$-th element from each of the $r$ words of length $\alpha_i$ and groups them into an $r$-tuple. This forms the $k$-th "super-letter" of the new $\mathbf{C}_{\alpha}$-structure.
	Formally, $T_U$ maps $s$ to:
	\begin{equation*}
		T_U(s) = \prod_{i=1}^m \overline{ \left[\left(\prod_{j=1}^r w_{j,i}^1\right) \left(\prod_{j=1}^r w_{j,i}^2\right) \ldots \left(\prod_{j=1}^r w_{j,i}^{\alpha_i}\right)\right] }
	\end{equation*}
	This map is clearly a bijection, as it simply regroups the elements of $U$.
	To prove $T$ is a natural isomorphism, we must verify its naturality. For any bijection $\sigma: U \to V$, we show that $T_V \circ \mathbf{C}_{\alpha^r}[\sigma] = (\mathbf{C}_{\alpha} \circ X^r)[\sigma] \circ T_U$.
	\noindent\textbf{LHS:} We first apply $\mathbf{C}_{\alpha^r}[\sigma]$ to $s$, which permutes all labels:
	$$
	\mathbf{C}_{\alpha^r}[\sigma](s) = \prod_{i=1}^m \left( \prod_{j=1}^r \overline{\left[\sigma(w_{j,i}^{1})\sigma(w_{j,i}^{2})\ldots \sigma(w_{j,i}^{\alpha_i})\right]} \right)
	$$
	Applying $T_V$ to this new structure in $V$ gives:
	$$
	T_V(\mathbf{C}_{\alpha^r}[\sigma](s)) = \prod_{i=1}^m \overline{\left[ \left(\prod_{j=1}^r \sigma(w_{j,i}^1)\right) \left(\prod_{j=1}^r \sigma(w_{j,i}^2)\right) \ldots \left(\prod_{j=1}^r \sigma(w_{j,i}^{\alpha_i})\right) \right]}
	$$
	
	\noindent\textbf{RHS:} We first apply $T_U$ to $s$:
	$$
	T_U(s) = \prod_{i=1}^m \overline{\left[ \left(\prod_{j=1}^r w_{j,i}^1\right) \left(\prod_{j=1}^r w_{j,i}^2\right) \ldots \left(\prod_{j=1}^r w_{j,i}^{\alpha_i}\right)\right] }
	$$
	The action $(\mathbf{C}_{\alpha} \circ X^r)[\sigma]$ permutes the underlying elements within each $r$-tuple (each $X^r$-structure):
	\begin{align*}
		(\mathbf{C}_{\alpha} \circ X^r)[\sigma](T_U(s)) &= \prod_{i=1}^m \overline{\left[ \sigma\left(\prod_{j=1}^r w_{j,i}^1\right) \sigma\left(\prod_{j=1}^r w_{j,i}^2\right) \ldots \sigma\left(\prod_{j=1}^r w_{j,i}^{\alpha_i}\right) \right]} \\
		&= \prod_{i=1}^m \overline{ \left[\left(\prod_{j=1}^r \sigma(w_{j,i}^1)\right) \left(\prod_{j=1}^r \sigma(w_{j,i}^2)\right) \ldots \left(\prod_{j=1}^r \sigma(w_{j,i}^{\alpha_i})\right) \right]}
	\end{align*}
	Since LHS = RHS, the transformation $T$ is natural. As $T_U$ is a bijection for all $U$, $T$ is a natural isomorphism.
\end{proof}
In terms of series, we have the following analogous result.
\begin{cor}
Let $\alpha = (\alpha_1,\alpha_2,\ldots,\alpha_m)$ be a partition and $r\in\mathbb{N}$. We have the isomorphism
\begin{equation*}
	\mathbf{C}_{\alpha}(\mathbf{z})\circ p_1^r = \mathbf{C}_{\alpha^r}(\mathbf{z}).
\end{equation*}
\end{cor}
\begin{proof}
This follows directly by passing to the cycle index series of the respective species.
\end{proof}
\subsection{Cyclic molecules of the second kind}
Let $\alpha = (i_1^{\alpha_1},i_2^{\alpha_2},\ldots, i_m^{\alpha_m})$ be a partition of $n$. We define the subgroup $G_{\alpha}$ of $S_n$ as
 \begin{align*}G_{\alpha} &= \langle \sigma_{i_m^{\alpha_{m}}} \rangle\times \langle \sigma_{{(i_{m-1})}^{\alpha_{m-1}}} \rangle \times \cdots \times \langle \sigma_{i_1^{\alpha_1}} \rangle.
 \end{align*}
For example, $G_{43322} = \langle(1,2,3,4)\rangle \times \langle (5,6,7)(8,9,10) \rangle \times \langle (11,12)(13,14)\rangle$.
\begin{defin}
Let $\alpha = (i_1^{\alpha_1},i_2^{\alpha_2},\ldots, i_m^{\alpha_m})$ be a partition of $n$. We define the species $\mathbf{K}_{\alpha}$ of structures invariant under the subgroup $G_{\alpha}$ as
$$
\mathbf{K}_{\alpha} = X^n/G_{\alpha}= \mathbf{C}_{{i_m}^{\alpha_m}}\cdot \mathbf{C}_{{m-1}^{\alpha_{m-1}}} \ldots \mathbf{C}_{{i_1}^{\alpha_1}}.
$$
\end{defin}
A $\mathbf{K}_{\alpha}$-structure is thus a sequence of $\mathbf{C}_{i_j^{\alpha_j}}$-structures. While it is common to represent a product of structures as a tuple, in this work we will separate these individual structures using a "$\cdot$" symbol. For example, here are two $\mathbf{K}_{2^{2} 4}$-structures:
\begin{itemize}
	\item $\overline{[23]}\,\overline{[15]} \cdot \overline{[4687]}$,
	\item $\overline{[18]}\, \overline{[27]} \cdot \overline{[6453]}$.
\end{itemize}
\begin{prop}
	We have 
	$$
	\mathbf{K}_{\alpha} = \mathbf{C}_{i_1}(X^{\alpha_1}) \cdot \mathbf{C}_{i_2}(X^{\alpha_2})\cdots \mathbf{C}_{i_m}(X^{\alpha_m}).
	$$
\end{prop}
\begin{proof}
Proposition \ref{iso1} ensures that $\mathbf{C}_{i_j^{\alpha_j}} = \mathbf{C}_{i_j}(X^{\alpha_j})$ for all $j\in[m]$. The result then follows by substituting these $\mathbf{C}_{i_j^{\alpha_j}}$ with $\mathbf{C}_{i_j}(X^{\alpha_j})$ in the definition of $\mathbf{K}_{\alpha}$.
\end{proof}
The following lemma establishes a key uniqueness property for these three families of molecular species, relating their isomorphism class to their cycle index.
\begin{lem}\label{specyc}
	For these families, we have the following equivalences:
	\begin{enumerate}
		\item $\mathbf{E}_{\lambda} = \mathbf{E}_{\mu} \; \Leftrightarrow\; \mathrm{Z}_{\mathbf{E}_{\lambda}} = \mathrm{Z}_{\mathbf{E}_{\mu}} \; \Leftrightarrow \; \lambda = \mu.$
		
		\item $\mathbf{C}_{\lambda} = \mathbf{C}_{\mu} \; \Leftrightarrow\; \mathrm{Z}_{\mathbf{C}_{\lambda}} = \mathrm{Z}_{\mathbf{C}_{\mu}} \; \Leftrightarrow \; \lambda = \mu.$
		
		\item $\mathbf{K}_{\lambda} = \mathbf{K}_{\mu} \; \Leftrightarrow\; \mathrm{Z}_{\mathbf{K}_{\lambda}} = \mathrm{Z}_{\mathbf{K}_{\mu}} \; \Leftrightarrow \; \lambda = \mu.$
		
		\item As a consequence, the number of distinct isomorphism classes (and distinct cycle indices) in each of the three sets $\{\mathbf{E}_{\alpha}\}_{\alpha\vdash n}$, $\{\mathbf{C}_{\alpha}\}_{\alpha\vdash n}$, and $\{\mathbf{K}_{\alpha}\}_{\alpha\vdash n}$ is equal to $p(n)$, the number of partitions of $n$.
	\end{enumerate}    
\end{lem}
\begin{proof}
	The cycle index $\mathrm{Z}_S$ of a species $S$ is an isomorphism invariant, so $S = T \implies \mathrm{Z}_S = \mathrm{Z}_T$. The non-trivial part of (1), (2), and (3) is the reverse implication. We rely on the standard result (Lemma \ref{espconj}), which states that for molecular species, $X^n/H = X^n/K \Leftrightarrow H \text{ and } K \text{ are conjugate subgroups}$.	
	\begin{enumerate}
		\item[\textbf{(1)}] For $\mathbf{E}_\lambda = X^n/S_\lambda$, we have $\mathbf{E}_\lambda = \mathbf{E}_\mu \Leftrightarrow S_\lambda \simeq S_\mu$ by Lemma \ref{espconj}. The Young subgroups $S_\lambda$ and $S_\mu$ are conjugate if and only if $\lambda = \mu$. Separately, the cycle indices $\mathrm{Z}_{\mathbf{E}_\lambda} = h_\lambda$ (the complete homogeneous symmetric functions) form a basis for symmetric functions of degree $n$, so $\mathrm{Z}_{\mathbf{E}_\lambda} = \mathrm{Z}_{\mathbf{E}_\mu} \Leftrightarrow h_\lambda = h_\mu \Leftrightarrow \lambda = \mu$. All three conditions are equivalent.		
		\item[\textbf{(2)}] For $\mathbf{C}_\lambda = X^n/\langle \sigma_\lambda \rangle$, we first show it is well-defined. If $\sigma$ and $\tau$ both have cycle type $\lambda$, they are conjugate, i.e., $\tau = g \sigma g^{-1}$ for some $g \in S_n$. This implies their generated cyclic subgroups are also conjugate: $\langle \tau \rangle = \langle g \sigma g^{-1} \rangle = g \langle \sigma \rangle g^{-1}$. By Lemma \ref{espconj}, $X^n/\langle \sigma \rangle = X^n/\langle \tau \rangle$. Thus, the isomorphism class $\mathbf{C}_\lambda$ depends only on $\lambda$.
		The proof of (1) shows $\mathrm{Z}_{\mathbf{C}_\lambda} = \mathrm{Z}_{\mathbf{C}_\mu} \Leftrightarrow \lambda = \mu$ is equivalent to $\mathbf{C}_\lambda = \mathbf{C}_\mu \Leftrightarrow \lambda = \mu$. This relies on the fact that the conjugacy class of a cyclic subgroup in $S_n$ is uniquely determined by the cycle type of its generators. Since there are $p(n)$ such cycle types, there are $p(n)$ such conjugacy classes.		
		\item[\textbf{(3)}] The proof for $\mathbf{K}_\lambda = X^n/G_{\sigma_\lambda}$ follows the same logic. We assume $G_{\sigma_\lambda}$ is well-defined, i.e., $\lambda(\sigma) = \lambda(\tau) \implies G_\sigma \simeq G_\tau$. Then, $\mathbf{K}_\lambda = \mathbf{K}_\mu \Leftrightarrow G_{\sigma_\lambda} \simeq G_{\sigma_\mu}$. We also assume this happens if and only if $\lambda = \mu$. This gives $\mathbf{K}_\lambda = \mathbf{K}_\mu \Leftrightarrow \lambda = \mu$, which implies $\mathrm{Z}_{\mathbf{K}_\lambda} = \mathrm{Z}_{\mathbf{K}_\mu} \Leftrightarrow \lambda = \mu$.
		\item[\textbf{(4)}] From (1), (2), and (3), we have shown that for each family, the map $\lambda \mapsto \mathrm{Z}_{\text{Species}_\lambda}$ is a bijection from the set of partitions of $n$ to the set of distinct cycle indices. Since the size of the set of partitions of $n$ is $p(n)$, the number of distinct cycle indices in each family is $p(n)$.
	\end{enumerate}
\end{proof}
In other words, equality between these species is equivalent to the equality of their cycle index series. In general, this property does not hold for arbitrary molecular species.

\section{Basis of symmetric functions}
In the previous section, we introduced the families of symmetric functions corresponding to the cycle index series $\{Z_{\mathbf{C}_{\alpha}}\}_{\alpha\vdash n}$ and $\{Z_{\mathbf{K}_{\alpha}}\}_{\alpha\vdash n}$. In this section, we prove that these families constitute $\mathbb{Q}$-basis for the space of homogeneous symmetric functions $\mathbb{Q}\otimes \Lambda_n$. Furthermore, we investigate the relationship between these new basis and the classical ones, specifically establishing the transition matrices to the power sum, monomial, complete homogeneous, and Schur functions.   
\begin{defin}
	Let $\alpha = (a_1, a_2, \ldots, a_k) = (i_1^{\alpha_1}, i_2^{\alpha_2}, \ldots, i_m^{\alpha_m})$ be a partition of an integer $n$. We associate three symmetric functions with $\alpha$, corresponding to the cycle index series of the set molecule, and the cyclic molecules of the first and second kind, defined as follows:
	\begin{itemize}
		\item The homogeneous symmetric function
		\[
		h_{\alpha}(\mathbf{z})
		:= Z_{\mathbf{E}_{\alpha}}
		= Z_{\mathbf{E}_{a_1}} \cdot Z_{\mathbf{E}_{a_2}} \cdots Z_{\mathbf{E}_{a_k}}.
		\]
		\item The cyclic symmetric function of the first kind
		\[
		\mathbf{C}_{\alpha}(\mathbf{z})
		:= Z_{\mathbf{C}_{\alpha}}
		= \frac{1}{o(\sigma)} \sum_{g \in \langle \sigma \rangle} p_{\lambda(g)}(\mathbf{z}) 
		\]
		where $\sigma$ is a permutation of cycle type $\alpha$ and $o(\sigma)$ denotes the order of the permutation $\sigma$.
		\item The cyclic symmetric function of the second kind
		\[
		\mathbf{K}_{\alpha}(\mathbf{z})
		:= Z_{\mathbf{K}_{\alpha}}
		= \mathbf{C}_{i_1^{\alpha_{1}}}(\mathbf{z}) \cdot \mathbf{C}_{i_2^{\alpha_{2}}}(\mathbf{z}) \cdots \mathbf{C}_{i_m^{\alpha_{m}}}(\mathbf{z}).
		\]
		
	\end{itemize}
\end{defin}
By Lemma \ref{espconj}, the symmetric function $\mathbf{C}_{\alpha}(\mathbf{z})$ does not depend on the choice of $\sigma$. It is worth noting that $\mathbf{K}_{\alpha}$ should not be confused with the Lyndon symmetric functions $\mathrm{L}_{\alpha}$ (see \cite{GesselReutenauer1993}); although related, they differ in their power-sum expansions: the former involves the Euler totient function $\phi$, while the latter involves the Möbius function $\mu$.
\begin{prop}
	We have 
	$$\mathbf{C}_{\alpha}(\mathbf{z})= \frac{1}{o(\sigma)}\sum_{k|o(\sigma)} \phi(k)p_{\alpha^{(o(\sigma)/k)}}(\mathbf{z})$$
where $\alpha^{(o(\sigma)/k)}$ denotes the cycle type of the permutation $\sigma^{o(\sigma)/k}$, a power of $\sigma$.
\end{prop}
\begin{proof}
The action of $\langle \sigma \rangle$ on the set of powers $\{\sigma^1,\sigma^2,\cdots,\sigma^{o(\sigma)}=\Id\}$ is equivalent to the action of the cyclic group generated by $(1,2,\cdots,o(\sigma))$ on the set $\{1,2,\ldots,o(\sigma)\}$. Recall that the cycle index series of $X^n/\langle (1,2,\ldots,o(\sigma))\rangle$ is given by
$$
\frac{1}{o(\sigma)}\sum_{h\in \langle (1,2,\ldots,o(\sigma))\rangle}p_{\lambda(h)}.
$$
Since these two actions are equivalent, the cycle index series of $X^n/\langle \sigma \rangle$ is obtained by substituting the subgroup $\langle (1,2,\ldots,o(\sigma))\rangle$ with $\langle \sigma\rangle$. Thus, we obtain:
\begin{align*}
\mathbf{C}_{\alpha}(\mathbf{z}) &=\frac{1}{o(\sigma)}\sum_{h\in \langle \sigma\rangle}p_{\lambda(h)}\\
&=\frac{1}{o(\sigma)}\sum_{k|o(\sigma)} \phi(k) p_{\alpha^{(o(\sigma)/k)}}.
\end{align*}
\end{proof}
\begin{remarque}
Given a permutation $\sigma$ of shape $\alpha$, one must carefully distinguish between the notation $\alpha^{(r)}$ and $\alpha^r$. The partition $\alpha^{(r)}$ denotes the shape of the permutation $\sigma^r$, whereas $\alpha^r$ denotes the partition obtained from $\alpha$ by repeating each part $r$ times.
\end{remarque}
\begin{prop}
	The three sets $\{h_{\alpha}(\mathbf{z})\}_{\alpha\vdash n}$, $\{\mathbf{C}_{\alpha}(\mathbf{z})\}_{\alpha\vdash n}$, and $\{\mathbf{K}_{\alpha}(\mathbf{z})\}_{\alpha\vdash n}$ form $\mathbb{Q}$-basis of $\Lambda_n\otimes \mathbb{Q}$, the space of homogeneous symmetric functions of degree $n$.
\end{prop}
\begin{proof}
	It is a standard result that $\{h_{\alpha}(\mathbf{z})\}_{\alpha\vdash n}$ forms a basis of $\Lambda_n$. We now establish this property for $\{\mathbf{C}_{\alpha}(\mathbf{z})\}_{\alpha\vdash n}$ and $\{\mathbf{K}_{\alpha}(\mathbf{z})\}_{\alpha\vdash n}$.
	Let $\sigma$ be a permutation of type $\alpha$. We have:
\begin{align*}
	\mathbf{C}_{\alpha}(\mathbf{z}) &= \frac{1}{o(\sigma)}\sum_{k|o(\sigma)} \phi(k)p_{\lambda(\sigma^{o(\sigma)/k})}(\mathbf{z})\\
	&=\frac{\phi(o(\sigma))}{o(\sigma)} p_{\alpha}(\mathbf{z}) + \frac{1}{o(\sigma)}\sum_{\mu< \alpha} a_{\mu} p_{\mu}(\mathbf{z}).
\end{align*}
Observe that the transition matrix from $\{\mathbf{C}_{\alpha}(\mathbf{z})\}_{\alpha\vdash n}$ to $\{p_{\alpha}\}_{\alpha\vdash n}$ is upper triangular with non-zero diagonal entries. Hence, $\{\mathbf{C}_{\alpha}(\mathbf{z})\}_{\alpha\vdash n}$ is a generating family of $\Lambda_n$. By Lemma \ref{specyc}, the cardinality of this family is equal to $p(n)$, which is the dimension of $\Lambda_n$. 
Similarly, we provide the proof for $\{\mathbf{K}_{\mu}(\mathbf{z})\}_{\mu\vdash n}$. 
We have:
\begin{align*}
\mathbf{K}_{\alpha}(\mathbf{z}) &= \mathbf{C}_{i_1^{\alpha_1}}(\mathbf{z}) \cdot \mathbf{C}_{i_2^{\alpha_2}}(\mathbf{z}) \cdots \mathbf{C}_{i_m^{\alpha_m}}(\mathbf{z})\\
&=\frac{1}{o(1^{\alpha_1})}p_{1^{\alpha_1}} \frac{1}{o(2^{\alpha_2})}\sum_{k|2}\phi(k)p_{(2^{\alpha_2})^{(2/k)}}\; \cdots \; \frac{1}{o(m^{\alpha_m})}\sum_{k|m}p_{(m^{\alpha_m})^{(m/k)}}\\
&=\frac{1}{o(1)o(2)\cdots o(m)} p_{\alpha} + \sum_{\mu < \alpha} a_{\mu} p_{\mu}. 
\end{align*}
We find that the transition matrix from $\{\mathbf{K}_{\alpha}(\mathbf{z})\}_{\alpha\vdash n}$ to $\{p_{\alpha}(\mathbf{z})\}_{\alpha\vdash n}$ is also upper triangular with non-zero coefficients on the diagonal. Hence, $\{\mathbf{K}_{\alpha}(\mathbf{z})\}_{\alpha\vdash n}$ is a generating family of $\Lambda_n$. According to Lemma \ref{specyc}, the cardinality of this family is $p(n)$, the dimension of $\Lambda_n$. 
\end{proof}
The following section investigates the relationship between these basis and other classical basis, such as $p_{\lambda}$, $m_{\lambda}$, and $s_{\lambda}$.
The transition matrices relating the complete homogeneous symmetric functions to the other standard basis are well-established and will be stated without proof.
\begin{prop}\cite{AMJR}
	$$
	h_{\mu}(\mathbf{z}) = \sum_{\lambda}NM_{\lambda,\mu}m_{\lambda}(\mathbf{z}).
	$$
	where ${NM}_{\lambda,\mu}$ is the number of non-negative integer matrices such that the sum of the $i$-th row is equal to $\lambda_i$ and the sum of the $j$-th column is equal to $\mu_j$.
\end{prop}
Expanding a symmetric function into Schur functions addresses a fundamental problem in representation theory, as it corresponds to the decomposition into irreducible components. The expansion of $h_{\mu}$ is a classical result. 
\begin{prop}\cite{BS}\cite{WFJ}
	$$
	h_{\mu}(\mathbf{z}) = \sum_{\lambda\vdash |\mu|} \mathbf{K}_{\lambda,\mu} s_{\lambda}(\mathbf{z}).
	$$
	where $\mathbf{K}_{\mu,\lambda}$ is the Kostka number.
\end{prop}
The next part of this section defines the transition matrices between the two new basis and the standard basis.
\begin{cor}\label{ctop}
	The transition matrix $(a_{\lambda,\mu})$ from $\{\mathbf{C}_{\alpha}(\mathbf{z})\}_{\alpha\vdash n}$ to $\{p_{\alpha}(\mathbf{z})\}_{\alpha\vdash n}$ is given by
	\begin{equation}
		a_{\lambda,\mu} = 
		\begin{cases}
			\frac{\phi(k)}{o(\mu)} \quad \mbox{ if } \quad \lambda = \mu^{(o(\mu)/k)} \quad \mbox{ where }\quad  k|o(\mu).\\
			0 \quad \mbox{ otherwise }.
		\end{cases}
	\end{equation}
\end{cor}
\begin{proof}
This follows immediately by extracting the coefficients from the expression $$\mathbf{C}_{\mu}(\mathbf{z}) = \frac{1}{o(\mu)}\sum\limits_{k|o(\mu)}\phi(k) p_{\mu^{(k)}}.$$ 
\end{proof}
\begin{ex}
$$
\mathbf{C}_{422}(\mathbf{z}) = \frac{1}{4} p_{11111111} + \frac{1}{4} p_{221111} + \frac{1}{2} p_{422}.
$$
\end{ex}
The expansion of these symmetric functions into monomial symmetric functions provides generating series for weighted colorings on $n$ points, in accordance with Pólya's theorem. 
\begin{lem}{\cite{AMJR}}
	The transition matrix from $\{p_{\alpha}(\mathbf{z})\}_{\alpha\vdash n}$ to $\{m_{\alpha}(\mathbf{z})\}_{\alpha\vdash n}$ is given by $(\mathcal{OB}_{\lambda,\mu})_{\lambda,\mu}$, that is:
	$$
	p_{\mu}(\mathbf{z}) = \sum_{\mu\vdash n} \mathcal{OB}_{\lambda,\mu} m_{\lambda}(\mathbf{z}).
	$$
	where $\mathcal{OB}_{\mu,\lambda}$ is the number of ordered tabloid bricks of content $\lambda$ and shape $\mu$.
\end{lem}
In order to obtain the matrix relating $\{\mathbf{C}_{\alpha}\}$ to $\{h_{\alpha}\}_{\alpha\vdash n}$, we state the result for the transition $\{p_{\alpha}\}_{\alpha\vdash n}$ to $\{h_{\alpha}\}_{\alpha\vdash n}$. 
Let $B_{\lambda,\mu}$, be the set of all possible Young diagrams of $\mu$ where the rows of $\mu$ are partitioned into "bricks" of lengths giving the integer partition $\lambda$. We define the weight of $T \in B_{\lambda,\mu}$, denoted $w(T)$, to be the product of the lengths of the bricks ending each row in $T$ and let $w(B_{\mu,\lambda}) = \sum\limits_{T\in B_{\lambda,\mu}} w(T).$
\begin{lem}{\cite{AMJR}}
	The transition matrix from $\{p_{\alpha}(\mathbf{z})\}_{\alpha\vdash n}$ to $\{h_{\alpha}(\mathbf{z})\}_{\alpha\vdash n}$ is given by $((-1)^{l(\mu)+l(\lambda)}w(B_{\lambda,\mu}))_{\lambda,\mu}$, that is:
	$$
	p_{\mu}(\mathbf{z}) = \sum_{\lambda\vdash n} (-1)^{l(\mu)+l(\lambda)}w(B_{\lambda,\mu}) h_{\lambda}(\mathbf{z})
	$$

\end{lem}
\begin{prop}
	The transition matrix from $\{\mathbf{C}_{\alpha}(\mathbf{z})\}_{\alpha\vdash n}$ to $\{m_{\alpha}(\mathbf{z})\}_{\alpha\vdash n}$ is given by
	$$
	q_{\lambda,\mu} = \frac{1}{o(\mu)}\sum_{k|o(\mu)} \phi(k) \mathcal{OB}_{\mu^{(o(\mu)/k)},\lambda}
	$$
\end{prop}
\begin{proof}
	By multiplying the transition matrix $(a_{\lambda,\mu})_{\lambda,\mu}$ from $\{\mathbf{C}_{\alpha}(\mathbf{z})\}_{\alpha\vdash n}$ to $\{p_{\alpha}(\mathbf{z})\}_{\alpha\vdash n}$ with the transition matrix $(\mathcal{OB}_{\lambda,\mu})_{\lambda,\mu}$ from $\{p_{\alpha}(\mathbf{z})\}_{\alpha\vdash n}$ to $\{m_{\alpha}(\mathbf{z})\}_{\alpha\vdash n}$, we obtain the transition matrix from $\{\mathbf{C}_{\alpha}(\mathbf{z})\}_{\alpha\vdash n}$ to $\{m_{\alpha}(\mathbf{z})\}_{\alpha\vdash n}$.
\end{proof}
\begin{cor}
	$$
	\mathbf{C}_{\mu}(\mathbf{z}) = \sum_{\lambda\vdash n} \frac{1}{o(\mu)}\sum_{k|o(\mu)} \phi(k) \mathcal{OB}_{\mu^{(o(\lambda)/k)},\lambda} m_{\lambda}.
	$$
\end{cor}
\begin{proof}
This is a direct consequence of the previous Proposition, obtained by a change of basis.
\end{proof}
\begin{ex}
\begin{align*}
\mathbf{C}_{42}(\mathbf{z}) &= 180 m_{111111} + 90 m_{21111} + 46 m_{2211} + 24 m_{222} + 30 m_{3111} + \\&16 m_{321} + 6 m_{33} + 8 m_{411} + 5 m_{42} + 2 m_{51} + m_{6}.
\end{align*}
\end{ex}
\begin{prop}\label{count_struct}
	The number of $\mathbf{C}_{\mu}$-structures on an $n$-element set is equal to 
	\begin{align*}
	\frac{1}{o(\mu)}\sum_{k|o(\mu)} \phi(k) \mathcal{OB}_{\mu^{(o(\lambda)/k)},1^n}&= \frac{n!}{o(\mu)}\\
	&= \frac{1}{o(\mu)}\sum_{g\in \langle \lambda\rangle} \mathcal{OB}_{g,1^n}.
	\end{align*}
\end{prop}
\begin{proof}
	By Pólya's theorem, the coefficient of $m_{1^n}$ in the cycle index series of a species yields the number of these structures. For the second expression, the coefficient of $p_{1^n}$ multiplied by $n!$ yields the count of these same structures.
\end{proof}
Consequently, we immediately obtain the following identity:
\begin{cor}
	If $\lambda\vdash n$, we have 
	$$
	\sum_{k|o(\lambda)} \phi(k) \mathcal{OB}_{\lambda^{(o(\lambda/k))}, 1^n} = n!.
	$$
\end{cor}
\begin{proof}
This follows directly from Proposition \ref{count_struct}.
\end{proof}
Furthermore, the cycle index series of species are always Schur-positive. Indeed, for a species $F : \mathbb{B} \to \Ens$, and for each $n \in \mathbb{N}$, the vector space $\mathrm{C}F[n]$ carries a natural representation of $S_n$.  
The Frobenius characteristic of this representation is precisely the cycle index series $Z_F(p_1, \dots, p_n)$. Thus, $Z_F(p_1, \dots, p_n)$ is Schur-positive, as it corresponds to the Frobenius characteristic of an $S_n$–module (see \cite{FB2}). The following proposition gives the explicit decomposition of $\mathbf{C}_{\alpha}(\mathbf{z})$ into Schur functions.
\begin{prop}
	The transition matrix $(d_{\lambda,\mu})_{\lambda,\mu}$ from $\{\mathbf{C}_{\alpha}\}_{\alpha\vdash n}$ to $\{s_{\alpha}\}_{\alpha\vdash n}$ is given by $(d_{\lambda,\mu})_{\lambda,\mu}$, where:
	\begin{align*}
		d_{\lambda,\mu} &= \frac{1}{o(\mu)}\sum_{k|o(\mu)} \phi(k) \chi_{\mu^{(o(\mu)/k)}}^{\lambda}\\
		&=\frac{1}{o(\mu)}\sum_{h\in\langle \sigma_\mu\rangle} \chi^{\lambda}(h).
	\end{align*}
	and $\chi_{\mu^{(o(\mu)/k)}}^{\lambda}$ is the character of the Specht module $S^{\lambda}$ evaluated at $\mu^{(o(\mu)/k)}$.
\end{prop}
\begin{proof}
	The transition matrix $(a_{\lambda,\mu})$ from $\{\mathbf{C}_{\alpha}(\mathbf{z})\}_{\alpha\vdash n}$ to $\{p_{\alpha}(\mathbf{z})\}_{\alpha\vdash n}$ is given in Corollary \ref{ctop}. It is a standard result that the transition matrix $(b_{\lambda,\mu})$ from $\{p_{\alpha}(\mathbf{z})\}_{\alpha\vdash n}$ to $\{s_{\alpha}(\mathbf{z})\}_{\alpha\vdash n}$ is given by $(\chi_{\lambda}^{\mu})_{\mu,\lambda}$, the character of the Specht module $S^{\mu}$ evaluated at $\lambda$. By multiplying $(a_{\lambda,\mu})(b_{\lambda,\mu})$, we obtain $(d_{\lambda,\mu})$.
\end{proof}
\begin{prop}
	We have
	\[
	\mathbf{C}_{\mu}(\mathbf{z}) = \Frob( \mathrm{ch}\big(\mathrm{Ind}_{\langle \sigma_{\mu}\rangle}^{S_n}\mathbf{1}_{\langle\sigma_mu\rangle}\big))
	\]
	and for every $\lambda\vdash n$, the multiplicity of $s_\lambda$ in $\mathbf{C}_\mu(\mathbf{z})$ is
	\[
	d_{\lambda,\mu}=\langle \mathbf{C}_\mu(\mathbf{z}),s_\lambda(\mathbf{z})\rangle
	=\frac{1}{o(\mu)}\sum_{h\in \langle\sigma_{\mu}\rangle}\chi^\lambda(h),
	\]
	where $\chi^\lambda$ is the irreducible character associated with $S^\mu$. In particular
	$d_{\lambda,\mu}\in\mathbb{Z}_{\ge0}$.
\end{prop}
\begin{proof}
	The first equality arises from the identification between the cycle index series of
	$X^n/G$ and the Frobenius characteristic of the permutation module on the cosets
	$S_n/G$. By Frobenius reciprocity,
	\[
	\langle \mathrm{Ind}_G^{S_n}\mathbf{1}_G , S^\lambda\rangle_{S_n}
	= \langle \mathbf{1}_G, \mathrm{Res}^{S_n}_G S^\lambda\rangle_G
	= \frac{1}{|G|}\sum_{h\in G}\chi^\lambda(h).
	\]
	The stated formula and the integer positivity follow by setting $G=\langle \sigma_{\mu}\rangle$.
\end{proof}
We consider the following example:
\begin{exemple}
$$
\mathbf{C}_{11111}(\mathbf{z}) = s_{5} + 4 s_{41} + 5 s_{32} + 6 s_{311} + 5 s_{221} + 4 s_{2111} + s_{11111},
$$
$$
\mathbf{C}_{41}(\mathbf{z}) = s_{5} + s_{41} + s_{32} + s_{311} + 2 s_{221} + s_{2111}.
$$	
\end{exemple}
We now turn our attention to the family $\{\mathbf{K}_{\alpha}(\mathbf{z})\}_{\alpha\vdash n}$.
\begin{prop}\label{k_to_p}
We have 
$$
\mathbf{K}_{\mu}(\mathbf{z}) = \prod_{j=1}^m\frac{1}{i_j}\sum_{\substack{(k_1,k_2,\cdots,k_m)\\ k_j|i_j}}\phi(k_1)\phi(k_2)\cdots\phi(k_m) p_{V^{(k_1,k_2,\ldots,k_m)}(\mu)}.
$$	
where 
$V^{(k_1,k_2,\ldots,k_m)}(\mu)=\left((i_1)^{(i_1/k_1)}\right)^{\mu_1}\left((i_2)^{(i_2/k_2)}\right)^{\mu_2}\cdots \left((i_m)^{(i_m/k_m)}\right)^{\mu_m}$.
\end{prop}
\begin{proof}
This is obtained through a direct calculation:
	\begin{align*}
		\mathbf{K}_{\mu}(\mathbf{z}) &= \mathbf{C}_{i_1^{\mu_1}}(\mathbf{z})\cdot \mathbf{C}_{i_2^{\mu_2}}(\mathbf{z}) \cdots \mathbf{C}_{i_m^{\mu_m}}(\mathbf{z})\\
		&=\frac{1}{i_1}\sum_{k|i_1}\phi(k)p_{(i_1^{\mu_1})^{(i_1/k)}}\cdot \frac{1}{i_2}\sum_{k|i_2}\phi(k)p_{(i_2^{\mu_2})^{(i_2/k)}} \cdots \frac{1}{i_m}\sum_{k|i_m}\phi(k)p_{(i_m^{\mu_m})^{(i_m/k)}}\\
		&=\prod_{j=1}^m\frac{1}{i_j}\sum_{\substack{(k_1,k_2,\cdots,k_m)\\ k_j|i_j}}\phi(k_1)\phi(k_2)\cdots\phi(k_m) p_{(i_1)^{(i_1/k_1)}}^{\mu_1}p_{(i_2)^{(i_2/k_2)}}^{\mu_2}\cdots p_{(i_m)^{(i_m/k_m)}}^{\mu_m}
	\end{align*}
To simplify, we denote the expression $\left(i_1)^{(i_1/k_1)}\right)^{\mu_1}\left((i_2)^{(i_2/k_2)}\right)^{\mu_2}\cdots \left((i_m)^{(i_m/k_m)}\right)^{\mu_m}$ by $V^{(k_1,k_2,\ldots,k_m)}(\mu)$. Hence, we obtain the final expression.
\end{proof}
\begin{cor}
The transition matrix $(w_{\lambda,\mu})_{\lambda,\mu}$ from $\{\mathbf{K}_{\alpha}(\mathbf{z})\}_{\alpha\vdash n}$ to $\{p_{\alpha}(\mathbf{z})\}_{\alpha\vdash n}$ is given by, if $\lambda = (i_1^{\lambda_1},i_2^{\lambda_2},\ldots, i_m^{\lambda_m})$,
\begin{equation*}
	w_{\lambda,\mu} =  
	\begin{cases}
	\prod_{j=1}^m\frac{\phi(k_j)}{i_j} \quad \mbox{ if } \lambda =V^{(k_1,k_2,\ldots,k_m)}(\mu) \quad \mbox{ and } \quad  k_j|i_j \mbox{ for all } j\in [m]\\
	0 \quad \mbox{ otherwise}.
	\end{cases}
\end{equation*}
\end{cor}
\begin{proof}
The result follows by extracting the transition matrix from Proposition \ref{k_to_p}.
\end{proof}
\begin{ex}
	\begin{align*}
	\mathbf{K}_{422}(\mathbf{z}) = \frac{1}{8} p_{11111111} + \frac{1}{4} p_{221111} + \frac{1}{8} p_{2222} + \frac{1}{4} p_{41111} + \frac{1}{4} p_{422}.	
	\end{align*}
\end{ex}
\begin{prop}
	The transition matrix from $\{\mathbf{K}_{\alpha}(\mathbf{z})\}_{\alpha\vdash n}$ to $\{m_{\alpha}(\mathbf{z})\}_{\alpha\vdash n}$ is given by
	$$
	l_{\lambda,\mu} = \prod_{j=1}^m\frac{1}{o(i_j)}\sum_{\substack{(k_1,k_2,\ldots,k_m)\\k_j|i_j}}\phi(k_1)\phi(k_2)\cdots \phi(k_m) \mathcal{OB}_{V^{(k_1,k_2,\ldots,k_m)(\mu),\lambda}}
	$$
\end{prop}
\begin{proof}
	By multiplying the transition matrix $(w_{\lambda,\mu})_{\lambda,\mu}$ from $\{\mathbf{K}_{\alpha}(\mathbf{z})\}_{\alpha\vdash n}$ to $\{p_{\alpha}(\mathbf{z})\}_{\alpha\vdash n}$ with the transition matrix $(\mathcal{OB}_{\lambda,\mu})_{\lambda,\mu}$ from $\{p_{\alpha}(\mathbf{z})\}_{\alpha\vdash n}$ to $\{m_{\alpha}(\mathbf{z})\}_{\alpha\vdash n}$, we obtain the transition matrix from $\{\mathbf{K}_{\alpha}(\mathbf{z})\}_{\alpha\vdash n}$ to $\{m_{\alpha}(\mathbf{z})\}_{\alpha\vdash n}$.
\end{proof}
We thus obtain the following expansion of $\mathbf{K}_{\mu}$ into the monomial basis $\{m_{\alpha}\}_{\alpha\vdash n}$
\begin{cor}
	$$
	\mathbf{K}_{\mu}(\mathbf{z}) = \prod_{j=1}^m \frac{1}{i_j}\sum_{\lambda\vdash n} \sum_{\substack{(k_1,k_2,\ldots,k_m)\\k_j|i_j}}\phi(k_1)\phi(k_2)\cdots \phi(k_2) \mathcal{OB}_{{V^{(k_1,k_2,\ldots,k_m)}(\mu)},\lambda} m_{\lambda}.
	$$
\end{cor}
\begin{proof}
	This is a direct consequence of the previous Proposition, obtained by a change of basis.
\end{proof}
For example, we have:
\begin{align*}
	\mathbf{K}_{42}(\mathbf{z}) &= 90 m_{111111} + 48 m_{21111} + 26 m_{2211} + 15 m_{222} + 18 m_{3111} + 10 m_{321} + \\&4 m_{33} + 6 m_{411} + 4 m_{42} + 2 m_{51} + m_{6}
\end{align*}
\begin{prop}
	The number of $\mathbf{K}_{\mu}$-structures on an $n$-element set is equal to 
	$$
	\prod_{j=1}^m \frac{1}{i_j}\sum_{\mu\vdash n} \sum_{\substack{(k_1,k_2,\ldots,k_m)\\k_j|i_j}}\phi(k_1)\phi(k_2)\cdots \phi(k_2) \mathcal{OB}_{{V^{(k_1,k_2,\ldots,k_m)}(\mu)},1^n}.
	$$
\end{prop}
\begin{proof}
	By Pólya's theorem, the coefficient of $m_{1^n}$ in the cycle index series of a species yields the number of these structures.
\end{proof}
\begin{prop}
	The transition matrix $(t_{\lambda,\mu})_{\lambda,\mu}$ from $\{\mathbf{K}_{\alpha}\}_{\alpha\vdash n}$ to $\{s_{\alpha}\}_{\alpha\vdash n}$ is given by:
	\begin{align*}
		t_{\lambda,\mu} &= \prod_{j=1}^m\frac{1}{i_j}\sum_{\substack{(k_1,k_2,\ldots,k_m)\\k_j|i_j}}\phi(k_1)\phi(k_2)\cdots \phi(k_2) \chi_{{V^{(k_1,k_2,\ldots,k_m)}(\mu)}}^{\lambda}
	\end{align*}
	where $\chi^{\lambda}_{\mu^{(o(\mu)/k)}}$ is the character of the Specht module $S^{\lambda}$ evaluated at $\mu^{(o(\mu)/k)}$.
\end{prop}
\begin{proof}
	The transition matrix $(a_{\lambda,\mu})$ from $\{\mathbf{C}_{\alpha}(\mathbf{z})\}_{\alpha\vdash n}$ to $\{p_{\alpha}(\mathbf{z})\}_{\alpha\vdash n}$ is given in Corollary \ref{ctop}. It is a standard result that the transition matrix $(b_{\lambda,\mu})$ from $\{p_{\alpha}(\mathbf{z})\}_{\alpha\vdash n}$ to $\{s_{\alpha}(\mathbf{z})\}_{\alpha\vdash n}$ is given by $(\chi_{\lambda}^{\mu})_{\mu,\lambda}$, the character of the Specht module $S^{\mu}$ evaluated at $\lambda$. Multiplying $(a_{\lambda,\mu})(b_{\lambda,\mu})$ yields $(t_{\lambda,\mu})$.
\end{proof}
\begin{cor}
	We have 
	$$
	\mathbf{K}_{\mu}(\mathbf{z}) = \sum_{\lambda\vdash n} t_{\lambda,\mu} s_{\lambda}(\mathbf{z}).
	$$
\end{cor}
\begin{ex}
$$
\mathbf{K}_{421}(\mathbf{z}) =s_{7} + 3 s_{52} + 2 s_{61} + 2 s_{43} + 2 s_{511} + 5 s_{421} + 2 s_{331} + 3 s_{322} + 4 s_{3211} + 2 s_{2221} + 2 s_{4111} + s_{31111} + s_{22111}.
$$
\end{ex}

\section{Extending Garsia–Remmel: Closure Theorems for cyclic molecules}
This section presents the principal contribution of this work: the extension of the Garsia–Remmel theorem to the newly constructed basis, establishing the closure of $\mathbf{C}_{\alpha}$ and $\mathbf{K}_{\alpha}$ under the Kronecker product. To achieve a comprehensive species-theoretic interpretation, we first introduce the subcategories $\mathrm{C}\mathrm{Esp}$ and $\mathrm{K}\mathrm{Esp}$ which house our new basis. We begin by reformulating the classical Garsia–Remmel result within this category framework and then proceed to establish the analogous decomposition theorems for the two cyclic basis.
\begin{defin}
	We define the category $\mathrm{H}\mathrm{Esp}$ whose objects are non-negative integer linear combinations of $\mathbf{E}_{\mu}$, i.e.,
	$$
	F = \sum_{\mu\vdash n} a_{\mu} \mathbf{E}_{\mu} \quad \mbox{ where $a_{\mu}\in\mathbb{N}$},
	$$
	and the morphisms are natural transformations (species morphisms).
\end{defin}
\begin{ex}\label{exdeco}
	An example is given by
	$$
	F = \mathbf{E}_{3} + 2\mathbf{E}_{21} + 3\mathbf{E}_{111}.
	$$
	This structure can be interpreted as a collection where the coefficient of each $\mathbf{E}_{\lambda}$ term represents the number of distinct colorings or labellings available for the $\mathbf{E}_{\lambda}$-structure. Specifically:
\begin{itemize}
\item The $\mathbf{E}_{3}$-structure (coefficient $1$) has unit multiplicity.
\item The $\mathbf{E}_{21}$-structure (coefficient $2$) admits two distinct decorations (e.g., \textit{blue} or \textit{red}).
\item The $\mathbf{E}_{111}$-structure (coefficient $3$) admits three distinct decorations (e.g., \textit{blue}, \textit{red}, or \textit{green}).
\end{itemize}
In general, the coefficient of $\mathbf{E}_{\lambda}$ acts as a multiplicity index, indicating the number of ways the corresponding abstract structure $\mathbf{E}_{\lambda}$ is realized or "decorated" within $F$.
\end{ex}
\begin{lem}
	Let $n$ be a positive integer. For all partitions $\alpha$ and $\beta$ of $n$, there exists a partition $\mu$ of $n$ such that
	$$
	S_{\alpha}\cap S_{\beta} = S_{\mu}.
	$$
\end{lem}
\begin{proof}
	Let $\sigma\in S_{\alpha}\cap S_{\beta}$. This means $\sigma$ leaves invariant both partitions of $[n]$ associated with $\alpha$ and $\beta$, denoted $\underline{n}^{\alpha}$ and $\underline{n}^{\beta}$. That is, for all $i,j\in [n]$,
	$$
	\sigma(\underline{n}_i^{\alpha})=\underline{n}_i^{\alpha} \quad \mbox{ and } \quad \sigma(\underline{n}_j^{\beta})=\underline{n}_j^{\beta}.
	$$
	Thus, $\sigma(\underline{n}_i^{\alpha}\cap \underline{n}_j^{\beta})=\underline{n}_i^{\alpha}\cap \underline{n}_j^{\beta}$. By defining $\mu$ as the partition formed by the sizes of these non-empty intersections, $\mu = (|\underline{n}_i^{\alpha}\cap \underline{n}_j^{\beta}|)$, we have $\sigma\in S_{\mu}$.
\end{proof}
\begin{prop}
	Let $n$ be a positive integer. For all partitions $\alpha$ and $\beta$ of $n$, we have
	$$
	\mathbf{E}_{\alpha}\times \mathbf{E}_{\beta} = \sum_{\mu\vdash n}e_{\alpha,\beta}^{\mu} \mathbf{E}_{\mu},
	$$
	where $e_{\alpha,\beta}^{\mu}=\#\{S_{\alpha}\tau S_{\beta} \in S_{\alpha}\backslash S_{\underline{n}}\slash S_{\beta}\; |\; S_{\alpha}\cap \tau S_{\beta}\tau^{-1} \text{ is conjugate to } S_{\mu}\}$.
\end{prop}
\begin{proof}
Applying a fundamental decomposition result (Proposition \ref{propprod} from the full manuscript), we have
	\begin{align*}
		\left(X^n/S_{\alpha}\right)\times \left(X^n/S_{\beta}\right) &= \sum\limits_{S_{\alpha}\tau S_{\beta} \in S_{\alpha}\backslash S_n\slash S_{\beta}} \left(X^n/S_{\alpha}\cap \tau S_{\beta}\tau^{-1}\right)\\
		&=\sum_{\mu\vdash n} \#\{S_{\alpha}\tau S_{\beta} \in S_{\alpha}\backslash S_n\slash S_{\beta} \; |\; S_{\alpha}\cap \tau S_{\beta}\tau^{-1}=S_{\mu}\}\left(X^n/S_{\mu}\right).
	\end{align*}
\end{proof}
\begin{cor}
	The category $\mathrm{HEsp}$ is closed under the Cartesian product (or Hadamard product).
\end{cor}
\begin{proof}
As the Cartesian product of species is commutative and distributive (see \cite{JOYAL19811}\cite{FGP}), we have
\begin{align*}
\sum_{\mu\vdash n} a_{\mu} \mathbf{E}_{\mu} \times \sum_{\mu\vdash n} b_{\mu} \mathbf{E}_{\mu} &= \sum_{\alpha,\beta\vdash n} (a_{\alpha} b_{\beta}) \mathbf{E}_{\alpha}\times \mathbf{E}_{\beta}\\
&=\sum_{\alpha,\beta\vdash n} (a_{\alpha} b_{\beta})\sum_{\mu\vdash n}e_{\alpha,\beta}^{\mu} \mathbf{E}_{\mu}\\
&=\sum_{\mu\vdash n}\sum_{\alpha,\beta\vdash n} (a_{\alpha} b_{\beta})e_{\alpha,\beta}^{\mu} \mathbf{E}_{\mu}.
\end{align*}
\end{proof}
\begin{defin}
	For every finite set $U$, we define the set $$\mathbf{E}_{\alpha,\beta}^{\mu}[U] = \{(s,t)\in (\mathbf{E}_{\alpha}\times \mathbf{E}_{\beta})[U] \mid \Aut(s) \cap\Aut(t)\simeq S_{\mu}\}.$$
\end{defin}
\begin{prop}
	The transformation $\mathbf{E}_{\alpha,\beta}^{\mu}:\mathbb{B}\to \Ens$ is a subspecies of the Cartesian product $\mathbf{E}_{\alpha}\times \mathbf{E}_{\beta}$, i.e.,
	$$
		\mathbf{E}_{\alpha,\beta}^{\mu} \subset \mathbf{E}_{\alpha}\times \mathbf{E}_{\beta}.
	$$
\end{prop}
\begin{proof}
For any finite set $U$, $\mathbf{E}_{\alpha,\beta}^{\mu}[U]\subset (\mathbf{E}_{\alpha}\times \mathbf{E}_{\beta})[U]$. We must prove that it is stable under the transport of structure. Let $f:U\to V$ be a bijection. For every $(s,t)\in \mathbf{E}_{\alpha,\beta}^{\mu}[U]$, we have
	$$
	\Aut(\mathbf{E}_{\alpha}[f]s) \simeq \Aut(s)\quad \mbox{ and }\quad  \Aut(\mathbf{E}_{\beta}[f]t)\simeq \Aut(t).
	$$
	It is clear that $(\mathbf{E}_{\alpha}[f]s,\mathbf{E}_{\beta}[f]t) \in (\mathbf{E}_{\alpha}\times \mathbf{E}_{\beta})[V]$ and that $$\Aut(\mathbf{E}_{\alpha}[f]s)\cap \Aut(\mathbf{E}_{\beta}[f]t) \simeq S_{\mu}.$$
	Therefore, $(\mathbf{E}_{\alpha}[f]s,\mathbf{E}_{\beta}[f]t)\in \mathbf{E}_{\alpha,\beta}^{\mu}[V]$. Thus, $\mathbf{E}_{\alpha,\beta}^{\mu}$ is a subspecies of $\mathbf{E}_{\alpha}\times \mathbf{E}_{\beta}$.
\end{proof}
Mackey's theorem is central to our proof.
\begin{lem}\label{thmackey}[Mackey's theorem for group action, see \cite{AK3}\cite{AK}]
	Let $G$ be a group acting transitively on sets $X$ and $Y$. Then, for every $x\in X$ and $y\in Y$, we have the bijection
	$$
	G\backslash\!\backslash (X\times Y) \to \Aut(x)\backslash G\slash \Aut(y) : G(x,gy)\mapsto \Aut(x) g \Aut(y).
	$$
\end{lem}
\begin{lem}
	Let $U$ be a finite set. For each coset $\lambda S_{\alpha}\in \left(X^n/S_{\alpha}\right)[U]$, the automorphism group $\Aut(\lambda S_{\alpha}) \simeq S_{\alpha}$.
\end{lem}
\begin{proof}
We have $\Aut(\lambda S_{\alpha}) = \lambda S_{\alpha}\lambda^{-1}\simeq S_{\alpha}$.
\end{proof}
Applying Mackey's theorem with $X = \mathbf{E}_{\alpha}[\underline{n}]$ and $Y=\mathbf{E}_{\beta}[\underline{n}]$, and $G=S_{\underline{n}}$, we have for every $x\in X$ and $y\in Y$, $\Aut(x) \simeq S_{\alpha}$ and $\Aut(y) \simeq S_{\beta}$. 
\begin{lem}
	We have the bijection:
	$$
	S_n\backslash\!\backslash (\mathbf{E}_{\alpha}[\underline{n}]\times \mathbf{E}_{\beta}[\underline{n}]) \to S_{\alpha} \backslash S_n \slash S_{\beta}.
	$$
\end{lem}
\begin{cor}
	The isomorphism types (or unlabeled structures) associated with $\mathbf{E}_{\alpha,\beta}^{\mu}$ are in bijection with $\{S_{\alpha}\tau S_{\beta} \in S_{\alpha}\backslash S_{\underline{n}}\slash S_{\beta}\; |\; S_{\alpha}\cap \tau S_{\beta} \tau^{-1} = S_{\mu}\}$.
\end{cor}
We denote by $M_{\alpha,\beta}$ the set of matrices with coefficients in $\mathbb{N}$ whose row sums are $\alpha_1,\alpha_2,\ldots$ and column sums are $\beta_1,\beta_2,\ldots$. Let $\mu$ be a partition of $n$; we denote by $M_{\alpha,\beta}^{\mu}$ the set of matrices $A\in M_{\alpha,\beta}$ whose entries, when sorted in decreasing order, form the partition $\mu$.
\begin{lem}\label{bij}[Corollary 4.3.8 \cite{AK3}]
	We have the bijection
	$$
	S_{\lambda}\backslash S_{\underline{n}} \slash S_{\mu} \xrightarrow{\theta} M_{\alpha,\beta}: S_{\alpha}\tau S_{\beta} \mapsto (z_{ij}),
	$$
	where $z_{ij}:=|\underline{n}_i^{\alpha} \cap\tau \underline{n}_j^{\beta}|$.
\end{lem}
We define the type of the double coset $S_{\alpha}\tau S_{\beta}$ by $w(S_{\alpha}\tau S_{\beta})=\mu $ if $ S_{\alpha}\cap \tau S_{\beta}\tau^{-1}=S_{\mu}$ and the type of the matrix $Z=(z_{ij})\in M_{\alpha,\beta}$, $t(Z)=\mu$ if the decreasing sequence of the entries $(z_{ij})_{i,j}$ gives the partition $\mu$.
\begin{lem}
	The transformation $\theta$ preserves the types $t$ and $w$, i.e.,
	$$
	t(\theta(S_{\alpha}\tau S_{\beta})) = w(S_{\alpha}\tau S_{\beta}).
	$$
\end{lem}
\begin{proof}
	Suppose that $w(S_{\alpha}\tau S_{\beta})=\mu$, that is $S_{\alpha}\cap \tau S_{\beta}\tau^{-1} = S_{\mu}$. This implies that the multiset of sizes of the intersecting blocks $|\underline{n}_i^{\alpha}\cap \tau \underline{n}_j^{\beta}|$ is precisely the multiset of parts of $\mu$. Thus, the decreasing sequence of $|\underline{n}_i^{\alpha} \cap\tau \underline{n}_j^{\beta}|$ is equal to $\mu$. So $t(\theta(S_{\alpha}\tau S_{\beta}))=\mu = w(S_{\alpha}\tau S_{\beta})$.
\end{proof}
\begin{prop}
	The set $\{S_{\alpha}\tau S_{\beta} \in S_{\alpha}\backslash S_{\underline{n}}\slash S_{\beta}\; |\; S_{\alpha}\cap \tau S_{\beta}\tau^{-1} = S_{\mu}\}$ is in bijection with the set of matrices $M_{\alpha,\beta}^{\mu}$.
\end{prop}
\begin{proof}
	The map $\theta$ defined in Lemma \ref{bij} induces a bijection from $\{S_{\alpha}\tau S_{\beta} \in S_{\alpha}\backslash S_{\underline{n}}\slash S_{\beta}\; |\; S_{\alpha}\cap \tau S_{\beta}\tau^{-1} = S_{\mu}\}$ to $NM_{\alpha,\beta}^{\mu}$.
\end{proof}
\begin{ex}
Here are examples of the expansion:
$$
h_{42}\star h_{33} = h_{2211} + 2 h_{321},
$$
$$
h_{44}\star h_{332} = 2 h_{221111} + 2 h_{22211} + 4 h_{3221} + 2 h_{3311}.
$$
\end{ex}
\begin{cor}[A.M. Garsia and J. Remmel \cite{GR}\cite{GR}]
The coefficient of $h_{\mu}$ in the product $h_{\alpha}\star h_{\beta}$ is equal to the cardinality of $NM_{\alpha,\beta}^{\mu}$.
\end{cor}
In the following, similar properties are developed for the family $\{\mathbf{C}_{\alpha}(\mathbf{z})\}_{\alpha\vdash n}$. For two partitions ${\alpha}$ and ${\beta}$, the decomposition of the Hadamard product is given by:
$$
\mathbf{C}_{\alpha}(\mathbf{z})\star \mathbf{C}_{\beta}(\mathbf{z}) \,=\, \sum_{\mu\vdash n} b_{\alpha,\beta}^{\mu} \mathbf{C}_{\mu}(\mathbf{z}) \mbox{ where $b_{\alpha,\beta}^{\mu}\in\mathbb{N}$}.
$$
\begin{defin}
	We define the category $\mathrm{C}\mathrm{Esp}$ whose objects are non-negative integer linear combinations of the cyclic species $\mathbf{C}_{\mu}$, i.e.,
	$$
	F = \sum_{\mu\vdash n} a_{\mu} \mathbf{C}_{\mu} \quad \mbox{ where $a_{\mu}\in\mathbb{N}$},
	$$
	and whose morphisms are natural transformations (species morphisms).
\end{defin}
\begin{ex}
Let $F$ be an object of $\mathrm{C}\mathrm{Esp}$ defined by
$$
F = \mathbf{C}_{5} + 2\mathbf{C}_{32} + 2\mathbf{C}_{221}.
$$
This object can be interpreted in the same way as in Example \ref{exdeco}, where the coefficients represent the multiplicity or number of distinct decorations available for the corresponding $\mathbf{C}_{\mu}$-structure.
\end{ex}
As in the case of set molecules, we are interested in the following family of cyclic subgroups of the symmetric group $S_n$:
$$
\{\langle\sigma_\alpha\rangle\}_{\alpha\vdash n}.
$$
Here, $\langle\sigma_\alpha\rangle$ is the cyclic subgroup generated by the standard permutation $\sigma_\alpha$ of cycle type $\alpha$.
\begin{lem}
	Let $\sigma_\alpha$ and $\sigma_\beta$ be two permutations. The intersection of the corresponding cyclic groups is generated by a single permutation, i.e.,
	$$
	\langle \sigma_\alpha \rangle  \cap \langle \sigma_\beta \rangle = \langle \sigma_\mu \rangle,
	$$
	for some permutation $\sigma_\mu$.
\end{lem}
\begin{proof}\label{atom}
	The intersection $H = \langle \sigma_\alpha \rangle  \cap \langle \sigma_\beta \rangle$ is a subgroup of both $\langle \sigma_\alpha\rangle$ and $\langle \sigma_\beta\rangle$. Since every subgroup of a cyclic group is cyclic, $H$ must be cyclic. Thus, $H$ is generated by some element $\tau$. If $\langle \tau\rangle = H$, then $\tau\in \langle \sigma_\alpha \rangle $ and $\tau\in \langle \sigma_\beta \rangle$, meaning there exist integers $a$ and $b$ such that $\tau= \sigma_\alpha^a=\sigma_\beta^{b}$. Then $H$ is isomorphic to $\langle \sigma_\mu\rangle$ where $\mu$ is the shape of $\tau$.
\end{proof}
\begin{prop}\label{prop_C}
	Let $\alpha$ and $\beta$ be two partitions of $n$. We have the decomposition,
	$$
	\mathbf{C}_{\alpha}\times \mathbf{C}_{\beta} = \sum_{\mu\vdash n} \#\{\langle \sigma_\alpha \rangle \pi \langle \sigma_\beta \rangle \in \langle \sigma_\alpha \rangle\backslash S_n\slash \langle \sigma_\beta \rangle\,|\, \langle \sigma_\alpha \rangle \cap \pi\langle \sigma_\beta \rangle\pi^{-1} \text{ is conjugate to } \langle \sigma_\mu\rangle\} \left(X^n/\langle \sigma_\mu \rangle\right).
	$$
\end{prop}
\begin{proof}
Applying Proposition \ref{propprod} (Species Decomposition Theorem), we obtain
	$$
	X^n/\langle \sigma_\alpha \rangle\times X^n/\langle \sigma_\beta \rangle = \sum\limits_{\tau \in \langle \sigma_\alpha \rangle\backslash S_n \slash \langle \sigma_\beta \rangle} X^n/\left(\langle \sigma_\alpha \rangle\cap \tau \langle \sigma_\beta \rangle\tau^{-1}\right).
	$$
	Lemma \ref{atom} ensures that the intersection of two cyclic groups is cyclic. By grouping the double cosets whose resulting intersection is conjugate to $\langle \sigma_\mu \rangle$, we obtain the coefficient $b_{\alpha,\beta}^{\mu}$:
	$$
	X^n/\langle \sigma_\alpha \rangle\times X^n/\langle \sigma_\beta \rangle = \sum_{\mu\vdash n} \#\{\langle \sigma_\alpha \rangle \pi \langle \sigma_\beta \rangle \in \langle \sigma_\alpha \rangle\backslash S_n\slash \langle \sigma_\beta \rangle\,|\, \langle \sigma_\alpha \rangle \cap \pi\langle \sigma_\beta \rangle\pi^{-1}=\langle \sigma_\mu \rangle\} \left(X^n/\langle \sigma_\mu \rangle\right).
	$$
\end{proof}
\begin{defin}
	For every finite set $U$ of size $n$, we define
	$$
	\mathbf{C}_{\alpha,\beta}^{\mu}[U] = \{(s,t) \in (\mathbf{C}_{\alpha}\times \mathbf{C}_{\beta})[U]\; |\; \Aut(s) \cap \Aut(t)\simeq \langle \sigma_\mu \rangle\}.
	$$
\end{defin}
\begin{prop}
	The transformation $\mathbf{C}_{\alpha, \beta}^{\mu}: \mathbb{B}\to \Ens$ is a subspecies of the Hadamard product $\mathbf{C}_{\alpha}\times \mathbf{C}_{\beta}$.
\end{prop}
\begin{proof}
	We verify that $\mathbf{C}_{\alpha, \beta}^{\mu}$ is stable under the transport of structure. Let $f:U\to V$ be a bijection. For every $(s,t)\in \mathbf{C}_{\alpha,\beta}^{\mu}[U]$, since the automorphism groups are preserved under transport, we have
	$$
	\Aut(\mathbf{C}_{\alpha}[f]s) \simeq \Aut(s)\quad \mbox{ and }\quad  \Aut(\mathbf{C}_{\beta}[f]t)\simeq \Aut(t).
	$$
	It is clear that $(\mathbf{C}_{\alpha}[f]s,\mathbf{C}_{\beta}[f]t) \in (\mathbf{C}_{\alpha}\times \mathbf{C}_{\beta})[V]$. Furthermore, the intersection condition is preserved:
	$$\Aut(\mathbf{C}_{\alpha}[f]s)\cap \Aut(\mathbf{C}_{\beta}[f]t) \simeq \langle \sigma_\mu\rangle.$$
	Therefore, $(\mathbf{C}_{\alpha}[f]s,\mathbf{C}_{\beta}[f]t)\in \mathbf{C}_{\alpha,\beta}^{\mu}[V]$. This confirms that $\mathbf{C}_{\alpha,\beta}^{\mu}$ is a subspecies of $\mathbf{C}_{\alpha}\times \mathbf{C}_{\beta}$.
\end{proof}
Applying Mackey's theorem (Lemma \ref{thmackey}) for $X = \mathbf{C}_{\alpha}[\underline{n}]$ and $Y=\mathbf{C}_{\beta}[\underline{n}]$, with the group action $G=S_{\underline{n}}$, the automorphism groups are $\Aut(x) \simeq \langle \sigma_\alpha \rangle$ and $\Aut(y) \simeq \langle\sigma_\beta\rangle$. Hence, we get the following lemma as a consequence.
\begin{lem}\label{lemcycfirst}
	We have the bijection:
	$$
	S_n\backslash\!\backslash (\mathbf{C}_{\alpha}[\underline{n}]\times \mathbf{C}_{\beta}[\underline{n}]) \to \langle \sigma_\alpha \rangle \backslash S_n \slash \langle \sigma_\beta\rangle.
	$$
\end{lem}
\begin{prop}
	The coefficient $b_{\alpha,\beta}^{\mu}$ is equal to the number of unlabeled structures or isomorphism types of $\mathbf{C}_{\alpha,\beta}^{\mu}$.
\end{prop}
\begin{proof}
The bijection provided by Lemma \ref{lemcycfirst} shows that the number of unlabelled structures of $\mathbf{C}_{\alpha}\times \mathrm{C}_{\beta}$ is equal to the cardinality of the set of double cosets $\langle \sigma_\alpha\rangle\backslash S_n\slash \langle \sigma_\beta\rangle$. The coefficient $b_{\alpha,\beta}^{\mu}$ counts the number of double cosets $\langle \sigma_\alpha \rangle \pi \langle \sigma_\beta\rangle$ such that the stabilizer intersection $\langle \sigma_\alpha \rangle \cap \pi \langle \sigma_\beta\rangle\pi^{-1}$ is conjugate to $\langle \sigma_\mu \rangle$. This precisely corresponds to the number of isomorphism types in $\mathbf{C}_{\alpha}\times \mathbf{C}_{\beta}$ whose intersection of automorphism groups is $\langle \sigma_\mu\rangle$.
\end{proof}
\begin{cor}
	The category $\mathrm{CEsp}$ is closed under the Cartesian product (Hadamard product).
\end{cor}
\begin{proof}
	As the Cartesian product of species is commutative and distributive (see \cite{FGP}), we have
	\begin{align*}
		\sum_{\mu\vdash n} a_{\mu} \mathbf{C}_{\mu} \times \sum_{\mu\vdash n} b_{\mu} \mathbf{C}_{\mu} &= \sum_{\alpha,\beta\vdash n} (a_{\alpha} b_{\beta}) \mathbf{C}_{\alpha}\times \mathbf{C}_{\beta}\\
		&=\sum_{\alpha,\beta\vdash n} (a_{\alpha} b_{\beta})\sum_{\mu\vdash n}b_{\alpha,\beta}^{\mu} \mathbf{C}_{\mu}\\
		&=\sum_{\mu\vdash n}\sum_{\alpha,\beta\vdash n} (a_{\alpha} b_{\beta})b_{\alpha,\beta}^{\mu} \mathbf{C}_{\mu}.
	\end{align*}
\end{proof}
\begin{ex}
The following examples illustrate the decomposition for $n=6$:
	$$
	\mathbf{C}_{6}(\mathbf{z})\star \mathbf{C}_{33}(\mathbf{z}) = 38 \, \mathbf{C}_{111111}(\mathbf{z}) + 6 \, \mathbf{C}_{33}(\mathbf{z}).
	$$
	$$
	\mathbf{C}_{6}(\mathbf{z})\star \mathbf{C}_{6}(\mathbf{z}) = 18 \, \mathbf{C}_{111111}(\mathbf{z}) + 2 \, \mathbf{C}_{222}(\mathbf{z}) + 2 \, \mathbf{C}_{33}(\mathbf{z}) + 2 \, \mathbf{C}_{6}(\mathbf{z}).
	$$
	$$
	\mathbf{C}_{42}(\mathbf{z})\star \mathbf{C}_{6}(\mathbf{z}) = 30 \, \mathbf{C}_{111111}(\mathbf{z}).
	$$
\end{ex}
To the best of our knowledge, a systematic enumeration of the double cosets $\langle \sigma_\alpha\rangle \backslash S_n\slash \langle\sigma_\beta\rangle$ in terms of matrix-like combinatorial objects analogous to the A.M. Garsia and J. Remmel \cite{GR} matrices $\mathrm{NM}_{\alpha,\beta}^{\mu}$ is desirable, but remains an open problem.

\begin{prob}
	Find a direct combinatorial interpretation of the coefficients $b_{\alpha,\beta}^{\mu}$.
\end{prob}
Here is a list of computed coefficients for $n=6$:
\begin{itemize}
	\item $\mathbf{C}_{6}(\mathbf{z}) \star \mathbf{C}_{6}(\mathbf{z}) = 18 \, \mathbf{C}_{111111}(\mathbf{z}) + 2 \, \mathbf{C}_{222}(\mathbf{z}) + 2 \, \mathbf{C}_{33}(\mathbf{z}) + 2 \, \mathbf{C}_{6}(\mathbf{z})$
	\item $\mathbf{C}_{6}(\mathbf{z}) \star \mathbf{C}_{51}(\mathbf{z}) = 24 \, \mathbf{C}_{111111}(\mathbf{z})$
	\item $\mathbf{C}_{6}(\mathbf{z}) \star \mathbf{C}_{42}(\mathbf{z}) = 30 \, \mathbf{C}_{111111}(\mathbf{z})$
	\item $\mathbf{C}_{6}(\mathbf{z}) \star \mathbf{C}_{33}(\mathbf{z}) = 38 \, \mathbf{C}_{111111}(\mathbf{z}) + 6 \, \mathbf{C}_{33}(\mathbf{z})$
	\item $\mathbf{C}_{6}(\mathbf{z}) \star \mathbf{C}_{222}(\mathbf{z}) = 56 \, \mathbf{C}_{111111}(\mathbf{z}) + 8 \, \mathbf{C}_{222}(\mathbf{z})$
	\item $\mathbf{C}_{6}(\mathbf{z}) \star \mathbf{C}_{111111}(\mathbf{z}) = 120 \, \mathbf{C}_{111111}(\mathbf{z})$
	\item $\mathbf{C}_{51}(\mathbf{z}) \star \mathbf{C}_{51}(\mathbf{z}) = 28 \, \mathbf{C}_{111111}(\mathbf{z}) + 4 \, \mathbf{C}_{51}(\mathbf{z})$
	\item $\mathbf{C}_{42}(\mathbf{z}) \star \mathbf{C}_{42}(\mathbf{z}) = 44 \, \mathbf{C}_{111111}(\mathbf{z}) + 4 \, \mathbf{C}_{42}(\mathbf{z})$
	\item $\mathbf{C}_{42}(\mathbf{z}) \star \mathbf{C}_{411}(\mathbf{z}) = 44 \, \mathbf{C}_{111111}(\mathbf{z}) + 2 \, \mathbf{C}_{22 11}(\mathbf{z})$
	\item $\mathbf{C}_{33}(\mathbf{z}) \star \mathbf{C}_{33}(\mathbf{z}) = 76 \, \mathbf{C}_{111111}(\mathbf{z}) + 12 \, \mathbf{C}_{33}(\mathbf{z})$
	\item $\mathbf{C}_{321}(\mathbf{z}) \star \mathbf{C}_{321}(\mathbf{z}) = 18 \, \mathbf{C}_{111111}(\mathbf{z}) + 2 \, \mathbf{C}_{21111}(\mathbf{z}) + 2 \, \mathbf{C}_{31^3}(\mathbf{z}) + 2 \, \mathbf{C}_{321}(\mathbf{z})$.
\end{itemize}

In the following, a similar property is developed for the family $\{\mathbf{K}_{\alpha}(\mathbf{z})\}_{\alpha\vdash n}$. For two partitions ${\alpha}$ and ${\beta}$, the decomposition of the Hadamard product is given by:
$$
\mathbf{K}_{\alpha}(\mathbf{z})\star \mathbf{K}_{\beta}(\mathbf{z}) \,=\, \sum_{\mu\vdash n} j_{\alpha,\beta}^{\mu} \mathbf{K}_{\mu}(\mathbf{z}) \mbox{ where $j_{\alpha,\beta}^{\mu}\in\mathbb{N}$}.
$$
\begin{defin}
	We define the category $\mathrm{K}\mathrm{Esp}$ whose objects are non-negative integer linear combinations of the $\mathbf{K}_{\mu}$ species, i.e.,
	$$
	F = \sum_{\mu\vdash n} a_{\mu} \mathbf{K}_{\mu} \quad \mbox{ where $a_{\mu}\in\mathbb{N}$}.
	$$
	The morphisms in this category are natural transformations (species morphisms).
\end{defin}
\begin{ex}
	Let $F$ be an object of $\mathrm{K}\mathrm{Esp}$ defined by
	$$
	F = \mathbf{K}_{42} + 2\mathbf{K}_{33} + 2\mathbf{K}_{222}.
	$$
This object can be interpreted in the same way as in Example \ref{exdeco}, where the coefficients represent the multiplicity or number of distinct decorations available for the corresponding $\mathbf{K}_{\mu}$-structure.
\end{ex}
As before, we consider the family of subgroups of $S_n$ defined by the partitions $\alpha$ that we have seen previously:
$$
\{G_\alpha\}_{\alpha\vdash n}.
$$
\begin{lem}
	Let $\alpha$ and $\beta$ be two partitions of $n$. The intersection of the corresponding product groups, $G_{\alpha} \cap G_{\beta}$, is itself a group of the same form $G_{\mu}$, i.e.,
	$$
	G_\alpha \cap G_\beta = G_\mu.
	$$
\end{lem}
\begin{proof}
Let $\alpha = (i_1^{a_1}, \dots, i_m^{a_m})$ and $\beta = (j_1^{b_1}, \dots, j_k^{b_k})$ be two partitions of $n$. The groups $G_{\alpha}$ and $G_{\beta}$ are internal direct products of cyclic groups $A_i$ and $B_j$ acting on disjoint supports $X_i$ and $Y_j$, respectively.
$$ G_{\alpha} = \prod_{i=1}^m A_i, \quad G_{\beta} = \prod_{j=1}^k B_j. $$
Let $H = G_{\alpha} \cap G_{\beta}$. As the intersection of abelian subgroups, $H$ is also abelian.
For any $h \in H$, $h$ must stabilize the support partitions $\{X_i\}_{i=1}^m$ and $\{Y_j\}_{j=1}^k$. Consequently, $h$ must stabilize the refined partition $Z_{ij} = X_i \cap Y_j$.
Thus, $H$ decomposes as an internal direct product:
$$ H = \prod_{i,j} H_{ij}, \quad \text{where } H_{ij} = H \cap S_{Z_{ij}}. $$
The component $H_{ij}$ is a subgroup of $A_i$ and $B_j$ when restricted to $Z_{ij}$. Since $A_i$ and $B_j$ are cyclic groups (on their full supports), $H_{ij}$ must also be cyclic, say $H_{ij} = \langle \rho_{ij} \rangle$.
Since $h \in G_\alpha$, the action of $h$ on $X_i$ has all cycles of length dividing $i_i$. Similarly, the action of $h$ on $Y_j$ has all cycles of length dividing $j_j$. It follows that the cycle structure of $\rho_{ij}$ on $Z_{ij}$ must consist of cycles all having the same length, $l_{ij} = \text{ord}(\rho_{ij})$.
The group $H$ is therefore the direct product of disjoint cyclic groups $H = \prod_{i,j} \langle \rho_{ij} \rangle$. By grouping these cyclic factors according to their cycle length $L_v$, we show that $H$ has the structure $G_{\mu}$ for the partition $\mu$ derived from these cycle lengths.
\end{proof}
\begin{prop}
	Let $\alpha,\beta$ be two partitions of $n$. We have the decomposition,
	$$
	\mathbf{K}_{\alpha}\times \mathbf{K}_{\beta} = \sum_{\mu\vdash n} \#\{G_\alpha \pi G_\beta \in G_{\alpha}\backslash S_n\slash G_{\beta}\,|\, G_{\alpha} \cap \pi G_{\beta}\pi^{-1} \text{ is conjugate to } G_\mu\} \mathbf{K}_{\mu}.
	$$
\end{prop}
\begin{proof}
Proposition \ref{propprod} (Species Decomposition Theorem) assures us that
$$
X^n/G_{\alpha}\times X^n/G_{\beta} = \sum\limits_{\tau \in G_{\alpha}\backslash S_n \slash G_\beta} X^n/\left(G_{\alpha}\cap \tau G_{\beta}\tau^{-1}\right).
$$
The preceding lemma allows us to group the terms where the intersection $G_{\alpha}\cap \tau G_{\beta}\tau^{-1}$ is conjugate to $G_\mu$. Since $\mathbf{K}_{\mu} = X^n/G_{\mu}$, we obtain:
\begin{align*}
X^n/G_\alpha\times X^n/G_\beta &= \sum_{\mu\vdash n} \#\{G_\alpha \pi G_\beta \in G_{\alpha}\backslash S_n\slash G_{\beta}\,|\, G_{\alpha} \cap \pi G_{\beta}\pi^{-1}=G_\mu\} \left(X^n/G_{\mu}\right)\\
&= \sum_{\mu\vdash n} \#\{G_\alpha \pi G_\beta \in G_{\alpha}\backslash S_n\slash G_{\beta}\,|\, G_{\alpha} \cap \pi G_{\beta}\pi^{-1}=G_\mu\} \mathbf{K}_{\mu}.
\end{align*}
\end{proof}
\begin{defin}
For every finite set $U$ of size $n$, we define the set of pairs of structures with a fixed intersection of automorphism groups:
$$
\mathbf{K}_{\alpha,\beta}^{\mu}[U] = \{(s,t) \in (\mathbf{K}_{\alpha}\times \mathbf{K}_{\beta})[U]\; |\; \Aut(s) \cap \Aut(t)\simeq G_{\mu}\}.
$$
\end{defin}
\begin{prop}
	The transformation $\mathbf{K}_{\alpha, \beta}^{\mu}: \mathbb{B}\to \Ens$ is a subspecies of the Hadamard product $\mathbf{K}_{\alpha}\times \mathbf{K}_{\beta}$.
\end{prop}
\begin{proof}
	We verify stability under transport of structure. Let $f:U\to V$ be a bijection. For any $(s,t)\in \mathbf{K}_{\alpha,\beta}^{\mu}[U]$, the automorphism groups are preserved:
	$$
	\Aut(\mathbf{K}_{\alpha}[f]s) \simeq \Aut(s)\quad \mbox{ and }\quad  \Aut(\mathbf{K}_{\beta}[f]t)\simeq \Aut(t).
	$$
	It is clear that
	 \begin{itemize}
	 \item $(\mathbf{K}_{\alpha}[f]s,\mathbf{K}_{\beta}[f]t) \in (\mathbf{K}_{\alpha}\times \mathbf{K}_{\beta})[V]$ and,
	 \item $\Aut(\mathbf{K}_{\alpha}[f]s)\cap \Aut(\mathbf{K}_{\beta}[f]t) \simeq G_{\mu}$.
	 \end{itemize}
	Thus, $(\mathbf{K}_{\alpha}[f]s,\mathbf{K}_{\beta}[f]t)\in \mathbf{K}_{\alpha,\beta}^{\mu}[V]$, confirming that $\mathbf{K}_{\alpha,\beta}^{\mu}$ is a subspecies of $\mathbf{K}_{\alpha}\times \mathbf{K}_{\beta}$.
\end{proof}
Applying Mackey's theorem for $X = \mathbf{K}_{\alpha}[\underline{n}]$, $Y=\mathbf{K}_{\beta}[\underline{n}]$, and $G=S_{\underline{n}}$, the automorphism groups are $\Aut(x) \simeq G_\alpha$ and $\Aut(y) \simeq G_\beta$. hence, we get the following lemma.
\begin{lem}\label{lemcycsecond}
	We have the bijection:
	$$
	S_n\backslash\!\backslash (\mathbf{K}_{\alpha}[\underline{n}]\times \mathbf{K}_{\beta}[\underline{n}]) \to G_\alpha \backslash S_n \slash G_\beta.
	$$
\end{lem}
\begin{prop}
	The coefficient $j_{\alpha,\beta}^{\mu}$ is equal to the number of unlabeled structures or isomorphism types of $\mathbf{K}_{\alpha,\beta}^{\mu}$.
\end{prop}
\begin{proof}
The bijection in Lemma \ref{lemcycsecond} shows that the number of unlabelled structures of $\mathbf{K}_{\alpha}\times \mathrm{K}_{\beta}$ is equal to the cardinality of the double coset set $G_{\alpha}\backslash S_n\slash G_{\beta}$. The coefficient $j_{\alpha,\beta}^{\mu}$ counts the number of double cosets $G_{\alpha} \pi G_{\beta}$ such that the stabilizer intersection $G_{\alpha} \cap \pi G_{\beta}\pi^{-1}$ is conjugate to $G_{\mu}$. This establishes that the coefficient $j_{\alpha,\beta}^{\mu}$ is equal to the number of unlabeled structures $(s,t)$ of $\mathbf{K}_{\alpha}\times \mathbf{K}_{\beta}$ satisfying $\Aut(s)\cap \Aut(t) \simeq G_{\mu}$.
\end{proof}
\begin{cor}
	The category $\mathrm{KEsp}$ is closed under the Cartesian product (Hadamard product).
\end{cor}
\begin{proof}
	As the Cartesian product of species is commutative and distributive (see \cite{FGP}), we have
	\begin{align*}
		\sum_{\mu\vdash n} a_{\mu} \mathbf{K}_{\mu} \times \sum_{\mu\vdash n} b_{\mu} \mathbf{K}_{\mu} &= \sum_{\alpha,\beta\vdash n} (a_{\alpha} b_{\beta}) \mathbf{K}_{\alpha}\times \mathbf{K}_{\beta}\\
		&=\sum_{\alpha,\beta\vdash n} (a_{\alpha} b_{\beta})\sum_{\mu\vdash n}j_{\alpha,\beta}^{\mu} \mathbf{K}_{\mu}\\
		&=\sum_{\mu\vdash n}\sum_{\alpha,\beta\vdash n} (a_{\alpha} b_{\beta})j_{\alpha,\beta}^{\mu} \mathbf{K}_{\mu}.
	\end{align*}
\end{proof}
\begin{ex}
The following examples illustrate the decomposition for $n=6$:
	$$
	\mathbf{K}_{222}(\mathbf{z})\star \mathbf{K}_{6}(\mathbf{z}) = 56 \, \mathbf{K}_{111111}(\mathbf{z}) + 8 \, \mathbf{K}_{222}(\mathbf{z}),
	$$
	$$
	\mathbf{K}_{42}(\mathbf{z})\star \mathbf{K}_{6} = 14 \, \mathbf{K}_{111111}(\mathbf{z}) + 2 \, \mathbf{K}_{222}(\mathbf{z}).
	$$
\end{ex}
As in the previous case, to the best of our knowledge, a systematic enumeration of the double cosets $G_\alpha \backslash S_n\slash G_\beta$ in terms of matrix-like combinatorial objects remains an open problem.

\begin{prob}
	Find a direct combinatorial interpretation of the coefficients $j_{\alpha,\beta}^{\mu}$.
\end{prob}
Here are some computed examples for $n=6$:
\begin{itemize}
	\item $\mathbf{K}_{6}(\mathbf{z}) \star \mathbf{K}_{6}(\mathbf{z}) = 18 \, \mathbf{K}_{111111}(\mathbf{z}) + 2 \, \mathbf{K}_{222}(\mathbf{z}) + 2 \, \mathbf{K}_{33}(\mathbf{z}) + 2 \, \mathbf{K}_{6}(\mathbf{z})$
	\item $\mathbf{K}_{6}(\mathbf{z}) \star \mathbf{K}_{42}(\mathbf{z}) = 14 \, \mathbf{K}_{111111} (\mathbf{z})+ 2 \, \mathbf{K}_{222}(\mathbf{z})$
	\item $\mathbf{K}_{51}(\mathbf{z}) \star \mathbf{K}_{51}(\mathbf{z}) = 28 \, \mathbf{K}_{111111}(\mathbf{z}) + 4 \, \mathbf{K}_{51}(\mathbf{z})$
	\item $\mathbf{K}_{42}(\mathbf{z}) \star \mathbf{K}_{42}(\mathbf{z}) = 10 \, \mathbf{K}_{111111}(\mathbf{z}) + \mathbf{K}_{21111}(\mathbf{z}) + \mathbf{K}_{222}(\mathbf{z}) + 2 \, \mathbf{K}_{42}(\mathbf{z})$
	\item $\mathbf{K}_{42}(\mathbf{z}) \star \mathbf{K}_{411}(\mathbf{z}) = 22 \, \mathbf{K}_{111111}(\mathbf{z}) + 2 \, \mathbf{K}_{411}(\mathbf{z})$
	\item $\mathbf{K}_{33}(\mathbf{z}) \star \mathbf{K}_{33}(\mathbf{z}) = 76 \, \mathbf{K}_{111111}(\mathbf{z}) + 12 \, \mathbf{K}_{33}(\mathbf{z})$
	\item $\mathbf{K}_{321}(\mathbf{z}) \star \mathbf{K}_{42}(\mathbf{z}) = 14 \, \mathbf{K}_{111111}(\mathbf{z}) + 2 \, \mathbf{K}_{21111}(\mathbf{z})$
	\item $\mathbf{K}_{321}(\mathbf{z}) \star \mathbf{K}_{321}(\mathbf{z}) = 18 \, \mathbf{K}_{111111}(\mathbf{z}) + 2 \, \mathbf{K}_{21111}(\mathbf{z}) + 2 \, \mathbf{K}_{31^3}(\mathbf{z}) + 2 \, \mathbf{K}_{321}(\mathbf{z})$
	\item $\mathbf{K}_{222}(\mathbf{z}) \star \mathbf{K}_{222}(\mathbf{z}) = 168 \, \mathbf{K}_{111111}(\mathbf{z}) + 24 \, \mathbf{K}_{222}(\mathbf{z})$
	\item $\mathbf{K}_{21111}(\mathbf{z}) \star \mathbf{K}_{42}(\mathbf{z}) = 42 \, \mathbf{K}_{111111}(\mathbf{z}) + 6 \, \mathbf{K}_{21111}(\mathbf{z})$
	\item $\mathbf{K}_{111111}(\mathbf{z}) \star \mathbf{K}_{111111}(\mathbf{z}) = 720 \, \mathbf{K}_{111111}(\mathbf{z})$.
\end{itemize}
\section{Combinatorial Interpretation of the Coefficient \texorpdfstring{$b_{n,n}^{\mu}$}{b\_\{n,n\}\^mu}}

In this section, we provide a combinatorial interpretation of the structure 
constants $b_{n,n}^{\mu}$ arising in the Kronecker product $\mathbf{C}_n \star 
\mathbf{C}_n$, in terms of the Steggall patterns introduced in~\cite{STE}.

\subsection{Steggall Patterns}

\begin{defin}[Element and Pattern]
An \textbf{element} of size $n$ is a permutation $(P_1, P_2, \dots, P_n)$ of 
$\{1, 2, \dots, n\}$, representing $n$ marked cells on an $n \times n$ grid 
such that no two cells share the same row or column.

A \textbf{Steggall pattern} is the equivalence class of an element under the 
action of the group $\mathbb{Z}_n \times \mathbb{Z}_n$, generated by the 
following two operations:
\begin{enumerate}
    \item \textbf{Cyclic shift of indices} (horizontal translation on the 
    torus):
    \[
        (P_1, P_2, \dots, P_n) \sim (P_2, P_3, \dots, P_n, P_1).
    \]
    \item \textbf{Translation of values modulo $n$} (vertical translation on 
    the torus):
    \[
        (P_i) \sim (P_i + k \!\!\!\pmod{n}), \qquad k \in \mathbb{Z}.
    \]
\end{enumerate}
We denote by $\mathbf{SP}_n$ the set of all Steggall patterns of size $n$.
\end{defin}

\begin{rem}
Geometrically, a Steggall pattern is an orbit of the natural action of 
$\mathbb{Z}_n \times \mathbb{Z}_n$ on the set of permutations of $[n]$, 
corresponding to toroidal translations of the marked cells on the grid. Two 
elements belong to the same pattern if and only if one can be obtained from the 
other by simultaneously sliding all marked cells cyclically along rows and 
columns.
\end{rem}

\begin{ex}
For $n = 4$, Steggall showed that there are exactly three distinct patterns, 
whose canonical representatives are $(1,2,3,4)$, $(1,4,3,2)$, and $(1,2,4,3)$.
\begin{center}
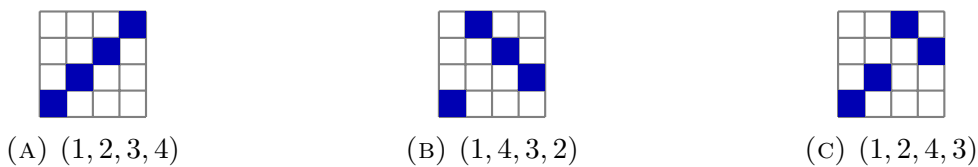
\begin{figure}[h!]
    \centering
    \begin{subfigure}[b]{0.3\textwidth}
        \centering
        \begin{tikzpicture}[scale=0.35]
            \draw[thick, gray] (0,0) grid (4,4);
            \fill[blue!70!black] (0,0) rectangle (1,1);
            \fill[blue!70!black] (1,1) rectangle (2,2);
            \fill[blue!70!black] (2,2) rectangle (3,3);
            \fill[blue!70!black] (3,3) rectangle (4,4);
        \end{tikzpicture}
        \caption{$(1, 2, 3, 4)$}
    \end{subfigure}
    \hfill
    \begin{subfigure}[b]{0.3\textwidth}
        \centering
        \begin{tikzpicture}[scale=0.35]
            \draw[thick, gray] (0,0) grid (4,4);
            \fill[blue!70!black] (0,0) rectangle (1,1);
            \fill[blue!70!black] (1,3) rectangle (2,4);
            \fill[blue!70!black] (2,2) rectangle (3,3);
            \fill[blue!70!black] (3,1) rectangle (4,2);
        \end{tikzpicture}
        \caption{$(1, 4, 3, 2)$}
    \end{subfigure}
    \hfill
    \begin{subfigure}[b]{0.3\textwidth}
        \centering
        \begin{tikzpicture}[scale=0.35]
            \draw[thick, gray] (0,0) grid (4,4);
            \fill[blue!70!black] (0,0) rectangle (1,1);
            \fill[blue!70!black] (1,1) rectangle (2,2);
            \fill[blue!70!black] (2,3) rectangle (3,4);
            \fill[blue!70!black] (3,2) rectangle (4,3);
        \end{tikzpicture}
        \caption{$(1, 2, 4, 3)$}
    \end{subfigure}
    \caption{The three Steggall patterns for $n = 4$. The first two are 
    affine patterns with stabiliser of size $d = 4$; the third is a generic 
    pattern with trivial stabiliser $d = 1$.}
    \label{fig:steggall_n4}
\end{figure}
\end{center}
\end{ex}

A pattern is the set of configurations obtained by “sliding” these blue squares cyclically on the grid (torus).

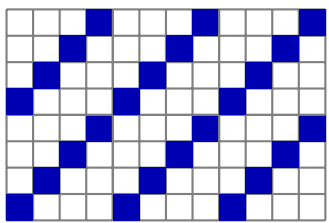
\begin{figure}[h]
\centering
\begin{tikzpicture}[scale=0.35]

\foreach \i in {0,1} {
\foreach \j in {0,1,2} {
\begin{scope}[shift={(4*\j,4*\i)}]
\draw[thick, gray] (0,0) grid (4,4);

        \fill[blue!70!black] (0,0) rectangle (1,1);
        \fill[blue!70!black] (1,1) rectangle (2,2);
        \fill[blue!70!black] (2,2) rectangle (3,3);
        \fill[blue!70!black] (3,3) rectangle (4,4);
    \end{scope}
}

}

\end{tikzpicture}
\caption{Pattern (1, 2, 3, 4)}
\end{figure}

\begin{figure}[h]
\centering
\begin{tikzpicture}[scale=0.35]

\foreach \i in {0,1} {
\foreach \j in {0,1,2} {
\begin{scope}[shift={(4*\j,4*\i)}]
\draw[thick, gray] (0,0) grid (4,4);

        \fill[blue!70!black] (0,0) rectangle (1,1);
        \fill[blue!70!black] (1,3) rectangle (2,4);
        \fill[blue!70!black] (2,2) rectangle (3,3);
        \fill[blue!70!black] (3,1) rectangle (4,2);
    \end{scope}
}

}

\end{tikzpicture}
\caption{Pattern (1, 4, 3, 2)}
\end{figure}
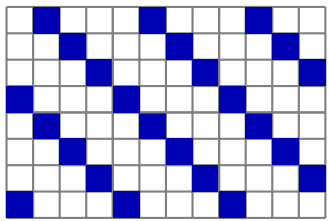

\begin{figure}[h]
\centering
\begin{tikzpicture}[scale=0.35]

\foreach \i in {0,1} {
\foreach \j in {0,1,2} {
\begin{scope}[shift={(4*\j,4*\i)}]
\draw[thick, gray] (0,0) grid (4,4);

        \fill[blue!70!black] (0,0) rectangle (1,1);
        \fill[blue!70!black] (1,1) rectangle (2,2);
        \fill[blue!70!black] (2,3) rectangle (3,4);
        \fill[blue!70!black] (3,2) rectangle (4,3);
    \end{scope}
}

}

\end{tikzpicture}
\caption{Pattern (1, 2, 4, 3)}
\end{figure}
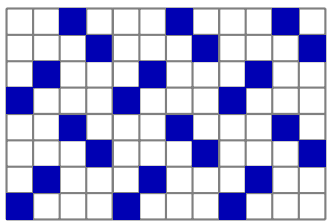

\subsection{The Stabiliser of a Pattern}

For a pattern $P \in \mathbf{SP}_n$, its \textbf{stabiliser} is the subgroup 
of $\mathbb{Z}_n \times \mathbb{Z}_n$ fixing $P$:
\[
    \mathrm{Stab}(P) 
    := \{(r, k) \in \mathbb{Z}_n \times \mathbb{Z}_n \mid 
        (P_{i+r} + k \!\!\!\pmod{n})_{i} = (P_i)_i\}.
\]
By the orbit-stabiliser theorem, the orbit of any element representing $P$ has 
size $n^2 / |\mathrm{Stab}(P)|$. In particular, $|\mathrm{Stab}(P)|$ is always 
a divisor of $n^2$.

\begin{rem}
One can show that if $|\mathrm{Stab}(P)| = d$, then $d \mid n$ and 
$\mathrm{Stab}(P)$ is cyclic of order $d$, generated by the pair $(n/d,\, 
n/d \cdot a)$ for some $a$ coprime to $d$. In particular, the possible 
stabiliser sizes are exactly the divisors of $n$.
\end{rem}

\subsection{Cameron's Identification and the Main Result}

In his work on homogeneous permutations~\cite{Cameron_2002}, Cameron established 
that the number of isomorphism types of the species $\mathbf{C}_n \times 
\mathbf{C}_n$ coincides with the number of Steggall patterns $|\mathbf{SP}_n|$, 
and that both are counted by sequence~\seqnum{A002619} in the OEIS:
\[
    1,\; 1,\; 2,\; 3,\; 8,\; 24,\; 108,\; 640,\; 4492,\; 36336,\; 
    329900,\; 3326788,\; 36846288,\; 444790512,\;\ldots
\]
More precisely, he showed that the action of $\mathbb{Z}_n \times \mathbb{Z}_n$ 
on $\mathfrak{S}_n$ by simultaneous cyclic shift of indices and translation of 
values is equivalent to the action defining the Steggall equivalence, yielding a 
canonical bijection:
\begin{equation}
\label{eq:cameron_bij}
    \mathbb{Z}_n \backslash \mathfrak{S}_n / \mathbb{Z}_n 
    \;\xrightarrow{\;\sim\;}\; \mathbf{SP}_n, 
    \qquad \mathbb{Z}_n \pi \mathbb{Z}_n \mapsto P_\pi.
\end{equation}

We can now state and prove the main result of this section, which gives an 
explicit combinatorial interpretation of the Kronecker structure constants 
$b_{n,n}^{\mu}$ in terms of Steggall patterns classified by their stabiliser 
size.

\begin{prop}
\label{thm:steggall_interp}
Let $n$ be a positive integer. The following identity holds in the category 
$\mathrm{CEsp}$:
\[
    \mathbf{C}_n \times \mathbf{C}_n 
    = \sum_{d \mid n} 
      \#\bigl\{P \in \mathbf{SP}_n \;\big|\; |\mathrm{Stab}(P)| = d\bigr\} 
      \cdot \mathbf{C}_{(d^{n/d})}.
\]
Equivalently, the Kronecker structure constants of the basis $\{C_\alpha(z)\}$ 
satisfy:
\[
    b_{n,n}^{\mu} = 
    \begin{cases} 
        \#\bigl\{P \in \mathbf{SP}_n \;\big|\; |\mathrm{Stab}(P)| = d\bigr\} 
        & \text{if } \mu = (d^{n/d}) \text{ for some } d \mid n, \\[4pt]
        0 
        & \text{otherwise.}
    \end{cases}
\]
\end{prop}

\begin{proof}
By Proposition~\ref{prop_C}, the Kronecker product $\mathbf{C}_n 
\times \mathbf{C}_n$ decomposes as:
\begin{equation}
\label{eq:CnCn_decomp}
    \mathbf{C}_n \times \mathbf{C}_n 
    = \sum_{\mu \vdash n} 
      \#\bigl\{\mathbb{Z}_n \pi \mathbb{Z}_n \in \mathbb{Z}_n \backslash 
      \mathfrak{S}_n / \mathbb{Z}_n 
      \;\big|\; 
      \mathbb{Z}_n \cap \pi \mathbb{Z}_n \pi^{-1} 
      \simeq \langle \sigma_\mu \rangle\bigr\} 
      \cdot \mathbf{C}_{\mu}.
\end{equation}

\medskip
\noindent\textit{Step 1: Restriction of admissible $\mu$.}

The condition $\mathbb{Z}_n \cap \pi \mathbb{Z}_n \pi^{-1} = \langle 
\sigma_\mu \rangle$ requires that $\sigma_\mu \in \mathbb{Z}_n$. Since 
$\mathbb{Z}_n = \langle \sigma_{(n)} \rangle$ is the cyclic group of order $n$ 
generated by the standard $n$-cycle $\sigma_{(n)}$, its subgroups are precisely 
$\langle \sigma_{(n)}^{n/d} \rangle$ for $d \mid n$, each cyclic of order $d$ 
with cycle type $(d^{n/d})$. Therefore $\mu$ must be of the form $(d^{n/d})$ 
for some divisor $d$ of $n$, and all other structure constants vanish.

\medskip
\noindent\textit{Step 2: Identification via Cameron's bijection.}

By Cameron's bijection~\eqref{eq:cameron_bij}, the double coset $\mathbb{Z}_n 
\pi \mathbb{Z}_n$ corresponds to a unique pattern $P_\pi \in \mathbf{SP}_n$. 
Under this correspondence, the stabiliser of the double coset within 
$\mathbb{Z}_n$ is:
\[
    \mathbb{Z}_n \cap \pi \mathbb{Z}_n \pi^{-1} 
    = \mathrm{Stab}_{\mathbb{Z}_n \times \mathbb{Z}_n}(P_\pi) \cap 
      (\mathbb{Z}_n \times \{0\}),
\]
which has order $d$ if and only if $|\mathrm{Stab}(P_\pi)| = d$. Therefore:
\[
    \#\bigl\{\mathbb{Z}_n \pi \mathbb{Z}_n 
    \;\big|\; 
    \mathbb{Z}_n \cap \pi \mathbb{Z}_n \pi^{-1} \simeq 
    \langle \sigma_{(d^{n/d})} \rangle\bigr\} 
    = \#\bigl\{P \in \mathbf{SP}_n \;\big|\; 
      |\mathrm{Stab}(P)| = d\bigr\}.
\]

\medskip
\noindent\textit{Step 3: Conclusion.}

Since every subgroup of $\mathbb{Z}_n$ is uniquely determined by its order, 
isomorphism and equality of subgroups of $\mathbb{Z}_n$ are equivalent. 
Substituting into~\eqref{eq:CnCn_decomp} and summing only over $d \mid n$ 
yields the result.
\end{proof}

\begin{cor}
The total number of Steggall patterns of size $n$ satisfies:
\[
    |\mathbf{SP}_n| 
    = \sum_{d \mid n} b_{n,n}^{(d^{n/d})} 
    = \frac{1}{n} \sum_{d \mid n} \varphi(d) \cdot \mathrm{Sq}(C_{[d^{n/d}]}),
\]
where the second equality follows from Burnside's lemma applied to the action 
of $\mathbb{Z}_n$ on $\mathfrak{S}_n / \mathbb{Z}_n$.
\end{cor}

\begin{ex}
For $n = 6$, the divisors of $6$ are $1, 2, 3, 6$. We have:
\[
    b_{6,6}^{(1^6)} 
    = \#\{P \in \mathbf{SP}_6 \mid |\mathrm{Stab}(P)| = 1\} = 18,
\]
\[
    b_{6,6}^{(2^3)} 
    = \#\{P \in \mathbf{SP}_6 \mid |\mathrm{Stab}(P)| = 2\} = 2,
\]
\[
    b_{6,6}^{(3^2)} 
    = \#\{P \in \mathbf{SP}_6 \mid |\mathrm{Stab}(P)| = 3\} = 2,
\]
\[
    b_{6,6}^{(6^1)} 
    = \#\{P \in \mathbf{SP}_6 \mid |\mathrm{Stab}(P)| = 6\} = 2,
\]
giving a total of $18 + 2 + 2 + 2 = 24$ distinct patterns, consistent with the 
known value $|\mathbf{SP}_6| = 24$. Moreover, these values match precisely the 
decomposition $C_6(z) \star C_6(z) = 18\,C_{1^6}(z) + 2\,C_{2^3}(z) + 
2\,C_{3^2}(z) + 2\,C_6(z)$ computed in Example~10.
\end{ex}
\section{Conclusion}
This work establishes a significant framework within algebraic combinatorics by introducing two novel bases for the ring of symmetric functions, denoted $\{\mathbf{C}_{\alpha}(\mathbf{z})\}_{\alpha\vdash n}$ and $\{\mathbf{K}_{\alpha}(\mathbf{z})\}_{\alpha\vdash n}$. Derived from the cycle index series of specialized combinatorial structures—specifically, the cyclic species of the first and second kind—these families are rigorously proven to form bases of $\Lambda_n$.
The central contribution of this paper is the study of Kronecker stability within the subcategories $\mathrm{CEsp}$ and $\mathrm{KEsp}$. By leveraging the theory of combinatorial species and analyzing the subgroups $\langle \sigma_\alpha\rangle$ and $G_{\alpha}$, we establish natural isomorphisms of species yielding the following integral and positive decomposition formulas:
$$
\mathbf{C}_{\alpha} \star \mathbf{C}_{\beta} = \sum_{\mu}b_{\alpha,\beta}^{\mu}\mathbf{C}_{\mu}
\quad \text{and} \quad
\mathbf{K}_{\alpha} \star \mathbf{K}_{\beta} = \sum_{\mu}j_{\alpha,\beta}^{\mu}\mathbf{K}_{\mu}.
$$
This work revisits the foundational results of A.M. Garsia and J. Remmel \cite{GR} on the Kronecker coefficients of $h_{\alpha} \star h_{\beta}$, and positions the bases $\mathbf{C}_{\alpha}$ and $\mathbf{K}_{\alpha}$ as natural extensions of this framework.
A key open problem in this context is the existence of a fully combinatorial interpretation of the Kronecker product in these bases. As a partial answer to this question, we provide a combinatorial interpretation of the structure underlying the direct product $\mathbf{C}_n \times \mathbf{C}_n$, described in terms of the action of cyclic groups on the associated pattern sets. This construction reveals a concrete model for the orbit-stabilizer decomposition in this setting and captures part of the combinatorial structure governing Kronecker interactions in the $\mathbf{C}$-basis.
This research opens two primary avenues for future exploration:
\begin{enumerate}
	\item \textbf{Combinatorial Interpretation:} The most immediate challenge is to determine explicit combinatorial interpretations for the coefficients $b_{\alpha,\beta}^{\mu}$ and $j_{\alpha,\beta}^{\mu}$, analogous to the classical problem for Schur functions.
	\item \textbf{Non-Commutative Extensions:} A further fruitful direction is the investigation of non-commutative versions of these symmetric functions and the exploration of their associated descent algebras. The expansion of the $\mathbf{C}_{\lambda}$ basis into the ribbon function basis $\{r_{\mu}\}_{\mu \models n}$ is given by
	$$
	\mathbf{C}_{\lambda} = \sum_{\mu \models n} \sigma_{\lambda,\mu} r_{\mu},
	$$
	where the coefficients $\sigma_{\lambda,\mu}$ are non-negative integers defined by the algebraic formula:
	$$
	\sigma_{\lambda,\mu} = \sum_{\alpha \vdash n} d_{\lambda,\alpha} K_{\alpha,\mu} \in \mathbb{N}.
	$$
\end{enumerate}
\section{Acknowledgements}
We express our sincere gratitude to François Bergeron for his invaluable supervision of the first author's doctoral research. We are also indebted to Bérénice Delcroix-Oger and Matthieu Josuat-Vergès for their thoughtful feedback and helpful suggestions, which significantly improved this manuscript. This work would not have been possible without the generous financial support of the International Mathematical Union (IMU) through the Graduate Assistantship in Developing Countries (GRAID) program, administered by the Commission for Developing Countries (CDC).
\bibliography{paper}
\bibliographystyle{amsplain}
\end{document}